\documentclass[11pt,a4paper,reqno]{amsart}
\usepackage{color}
\usepackage{amsmath}
\usepackage{bm}
\usepackage{amsthm}
\usepackage{amsfonts}
\usepackage{amssymb}
\usepackage{enumerate}
\usepackage{graphicx}
\usepackage{comment}
\newcommand{\undd}[1]{\overline{\mathrm{d}}( #1 )}
\newcommand{\band}[1]{{\mathrm{d}^*}( #1 )}
\usepackage{color}

\newcommand{\abs}[1]{\left\lvert #1 \right\rvert}

\newcommand\R{{\mathbf{R}}}

\newcommand\subeq{\mathrel{%
  \ooalign{\raise0.2ex\hbox{$\subset$}\cr\hidewidth\raise-0.8ex\hbox{\scalebox{0.9}{$\sim$}}\hidewidth\cr}}}
  
\newcommand\F{{\mathbf{F}}}

\newcommand\calM{{\mathcal{M}}}

\newcommand\ord{{\mathrm{ord}}}

\newcommand\Z{{\mathbf{Z}}}
\newcommand\N{{\mathbf{N}}}
\newcommand\Q{{\mathbf{Q}}}
\newcommand\supp{{\mathrm{supp}}}

\newcommand\T{\mathbf{T}}

\newtheorem{theorem}{Theorem}

\newtheorem{lemma}[theorem]{Lemma}
\newtheorem{proposition}[theorem]{Proposition}
\newtheorem{corollary}[theorem]{Corollary}

\newtheorem*{proposition*}{Proposition}

\def\le{\leqslant}
\def\ge{\geqslant}
\usepackage{hyperref}
\usepackage{breakurl}

\theoremstyle{plain}

\theoremstyle{remark}
  \newtheorem{remark}{Remark}

\title{On additive bases in infinite abelian semigroups}

\author{Pierre-Yves Bienvenu}
\address{P.-Y. Bienvenu, School of Mathematics, Trinity College Dublin, College Green, Dublin 2}
\email{bienvenp@tcd.ie}

\author{Benjamin Girard} 
\address{B. Girard, Sorbonne Universit\'e and Universit\'e de Paris, CNRS, IMJ-PRG, F-75006 Paris, France}
\email{\texttt{benjamin.girard@imj-prg.fr}}

\author{Th\'ai Ho\`ang L\^e}
\address{T. H. L\^e, Department of Mathematics, The University of Mississippi, University, MS 38677, United States}
\email{\texttt{leth@olemiss.edu}}


\begin{document}

\maketitle

\begin{abstract} 
Building on previous work by Lambert, Plagne and the third author, we study various aspects of the behavior of additive bases in infinite abelian groups and semigroups.
We show that, for every infinite abelian group $T$, the number of  essential subsets of any additive basis is finite, and also that the number of essential subsets of cardinality $k$ contained in an additive basis of order at most $h$ can be bounded in terms of $h$ and $k$ alone. 
These results extend the reach of two theorems, one due to Deschamps and Farhi and the other to Hegarty, bearing upon $\N$.
Also, using invariant means, we address a classical problem, initiated by Erd\H{o}s and Graham and then generalized by Nash and Nathanson both in the case of $\N$, of estimating the maximal order $X_T(h,k)$ that a basis of cocardinality $k$ contained in an additive basis of order at most $h$ can have.  
Among other results, we prove that $X_T(h,k)=O(h^{2k+1})$ for every integer $k \ge 1$. This result is new even in the case where $k=1$. Besides the maximal order $X_T(h,k)$, the typical order $S_T(h,k)$ is also studied.
Our methods actually apply to a wider class of infinite abelian semigroups, thus unifying
in a single axiomatic frame the theory of additive bases in $\N$ and in abelian groups.
\end{abstract}

\section{Introduction}
Let $(T,+)$ be an abelian semigroup. 
If $A,B$ are two subsets of $T$ whose symmetric difference is finite, we write $A \sim B$. 
Also if $A\setminus B$ is finite, we write $A\subeq B$.
Further, the Minkowski sum of $A$ and $B$ is defined as $\{a+b : (a,b)\in A\times B\}$ and denoted by $A+B$.
For every integer $h \ge 1$, the Minkowski sum of $h$ copies of $A$ is denoted by $hA$. By $[h]$, we mean $\{ 1, 2, \ldots, h\}$.

A subset $A$ of $T$ is called an \em additive basis \em of $T$, or just a \em basis \em of $T$ for brevity, whenever there exists an integer $h \ge 1$ for which all but finitely many elements of $T$ can be represented as the sum of exactly $h$ (not necessarily distinct) elements of $A$. 
In other words, $A$ is a basis of $T$ if and only if $hA \sim T$ for some $h \ge 1$. 
Further we define $k\cdot A$ to be the set $\{kx : x\in A\}\subset kA$.
Thus $(2\cdot\N)+3$ is not a basis of $\N$, but $(2\cdot\N)\cup\{3\}$ is.
The smallest possible integer $h \ge 1$ in the  definition above is then denoted by $\ord_T^*(A)$ and is called the \em order \em of $A$ over $T$. 
If $A$ is not a basis of $T$, then we set $\ord_T^*(A) = \infty$.
Note that what we call a basis is sometimes referred to as a "basis with an exact order" (and our order
as the exact order).

The study of additive bases already has a rich history, especially in the special case where $T$ is the semigroup $\N$ of nonnegative integers; it originated in additive number theory, motivated by Goldbach-type problems, and became a topic of research in its own right; the reader is referred to the surveys \cite{grekos,p}.
Some of the most natural and widely open problems in the area happen to deal with the "robustness" of this notion, an active area of research at least since Erd\H{o}s and Graham \cite{eg2,eg}: what happens
when one removes a finite subset from a basis? Does it remain a basis, and if so what happens to the order of the basis? 
Lambert, Plagne and the third author \cite{llp} initiated the systematic study of these questions  in 
general infinite abelian groups, 
when the removed subset is a singleton, and obtained partial results.
The present paper expands on these results, and generalizes them to arbitrary finite subsets.
Note that when $A$ is a basis of order $h$ of an infinite abelian group $G$ and $x\in A$, letting $B=A\setminus\{x\}$, the set $B'=B-x$ is a weak basis in the sense that $\bigcup_{i=0}^hiB'\sim G$.
This property, which was systematically used in  \cite{llp} to derive properties of $B$, vanishes when
one removes more than one element.

We now proceed to describe our results. 

\subsection{Essential subsets and the function $E_T(h,k)$}
Let $A$ be an additive basis of $T$.
A subset $F \subset A$ such that $A \setminus F$ is no longer an additive basis of $T$ is called an \em exceptional \em subset of $A$. Observe that any subset of $A$ containing an exceptional subset of $A$ is exceptional itself. 
This last observation motivates the following definition. 
An exceptional subset which is minimal with respect to inclusion will be called an \em essentiality \em of $A$.
A finite essentiality is called an \emph{essential subset}.
For instance $A=\{1,2,3\}\cup (6\cdot\N)$ is a basis of order 3 of $\N$, where $\{1,2,6\}$ is exceptional but not essential, the essential subsets are
$\{1,3\}$ and $\{1,2\}$, whereas $\{2\}$ is not exceptional (but its removal increases the order).

This notion was introduced by Deschamps and Farhi and, in the special case where $T=\N$, they showed that the number of  essential subsets in any given basis must be finite \cite[Th\'eor\`eme 10]{df}. 
Lambert, Plagne and the third author proved that this holds in any infinite abelian group 
for essential subsets of cardinality one (also called exceptional elements) \cite{llp}.
We generalize this latter result to arbitrary essential subsets, thus proving a Deschamps-Farhi theorem in infinite abelian groups.

\begin{theorem} \label{th:e1}
Every basis of an infinite abelian group $G$
has finitely many essential subsets.
\end{theorem}

Deschamps-Farhi's method is specific to $\N$, so we develop a new argument using the 
quotients of the group. To put the theorem above into perspective, we recall that, as proved by Lambert, Plagne and the third author, additive bases abound in infinite abelian groups, since every such group admits at least one additive basis of every possible order $h \ge 1$ (see \cite[Theorem 1]{llp}).

Going back to the special case where $T=\N$, Deschamps and Farhi observed that, for every integer $h \ge 2$, additive bases of order at most $h$ can have an arbitrarily large number of essential subsets. However, the situation changes drastically when we restrict our attention to the number of essential subsets of cardinality $k$ that a basis of order at most $h$ can have.
Indeed, for any infinite abelian semigroup $(T,+)$ and any integers $h,k\ge 1$, let us define
$$E_T(h,k) = \max_{\substack{A \subset T\\hA\sim T}} |\{ F \subset A: F \textup{ is essential and } |F| = k\}|,$$
and set $E_T(h) = E_T(h,1)$.
We also introduce the variant $E_T(h,\le k)$, defined identically except that the condition
$|F| = k$ is relaxed into $|F| \le k$, so $E_T(h,k)\le E_T(h,\le k)$.

The function $E_{\N}(h)$ was introduced and first studied by Grekos \cite{grekos0} who proved that $E_{\N}(h) \le h-1$, which was later refined in \cite{dg}. 
For their part, Deschamps and Farhi asked if the function $E_{\N}$ took only finite values \cite[Probl\`eme 1]{df}. This was later confirmed by Hegarty \cite[Theorem 2.2]{hegarty1}, who went on and obtained several asymptotic results such as 
\begin{equation} \label{eq:e1}
 E_{\N}(h,k) \sim (h-1) \frac{\log k}{\log \log k}
\end{equation}
for any fixed $h \ge 1$ as $k$ tends to infinity, and
\begin{equation} \label{eq:e2}
E_{\N}(h,k) \asymp_k \left( \frac{h^k}{\log h} \right)^{\frac{1}{k+1}}
\end{equation}
for any fixed $k \ge 1$ as $h$ tends to infinity \cite[Theorems 1.1 and 1.2]{hegarty2}. 
His results actually also hold for $E_\N(h,\le k)$.
However, it is still an open problem to know whether, for all $k \ge 1$, there exists a constant $c_{\N,k} >0$ such that 
$E_{\N}(h,k) \sim c_{\N,k}(h^k/\log h)^{1/(k+1)}$ 
as $h$ tends to infinity; so far it is only known for $k=1$ thanks to Plagne \cite{p3}.

In the framework of infinite abelian semigroups, far less is known concerning the function $E_T(h,k)$.
In \cite[Theorem 2]{llp}, Lambert, Plagne and the third author proved that $E_G(h) \le h-1$ for every infinite abelian group $G$ and every integer $h \ge 1$, and also that, as far as infinite abelian groups are concerned, this inequality is best possible for all $h \ge 1$. 
However, beyond this result and the fact that $E_G(1,k) = 0$ holds by definition, even the finiteness of $E_G(h,k)$ when $h,k \ge 2$ was left to be established.
We do so in this paper, even bounding $E_G(h,k)$ uniformly in $G$. 
We actually give three bounds, corresponding to the two asymptotic regimes where $h$ or $k$ is large, and the special case $h=2$.
\begin{theorem} \label{th:e2}
Let $G$ be an infinite abelian group.
\begin{enumerate}[(i)]
\item For any fixed $k\ge 1$, for any $h\ge 2$
the bound $E_G(h, k)\ll_k h^k$ holds.
\item For any fixed $h\ge 2$, 
for any $k\ge 2$ we have $E_G(h, k)\ll_h(k\log k)^{h-1}$.
\item For $h=2$, we have $E_G(2,k)\le 2k-1$.
\end{enumerate}
\end{theorem}
We will actually bound $E_G(h,\le k)$ which is at least as large as $E_G(h,k)$.
Our proof reveals the intimate link between the function $E_G$ and the set of finite quotients of $G$.
Thus our methods rely on the theory of finite abelian groups, including duality.
As our next theorem shows, there are no nontrivial universal lower bounds for $E_G$; thus, notwithstanding the just stated universal upper bounds, the function $E_G$ depends greatly on the structure of $G$, more precisely on its finite quotients.

\begin{theorem}\label{th:ElowBounds}
Let $G$ be an infinite abelian group.
\begin{enumerate}[(i)]
\item The function $E_G$ is trivial (i.e. $E_G(h,k) =0$ for all $h, k$) if, and only if, $G$
contains no proper finite index subgroups.
\item If $G$ admits $(\Z/2\Z)^d$ as quotient for any $d\ge 1$, then we have
$E_G(h,k)\ge (h-1)(2k-1)$ for any $h$ and infinitely many $k$.
\end{enumerate}
\end{theorem}
Therefore, $E_G=0$ whenever $G$ is a divisible group such as $\R$ or $\Q$ (i.e. for any $x\in G,n\in\N_{>0}$ there exists $y\in G$ such that $x=ny$), whereas  $E_G$ grows at least linearly in $h$ and in $k$ when $G$ is (the additive group of) 
$\Z^\N$ or $\F_2[t]$ for instance.
This contrasts sharply with $E_\N$ in view of the estimates \eqref{eq:e1} and \eqref{eq:e2}.
Also note that, when paired with Theorem \ref{th:e2} ($iii$), item $(ii)$ yields a good understanding of $E_G(2,k)$. 
Finally, the method we used to prove item $(ii)$ still applies when $(\Z/2\Z)^d$ is replaced by any finite abelian group (see Propositions \ref{prop:Equotients} and \ref{prop:decomposition}), even though the quality of the bound may no longer be close to optimal (see Proposition \ref{prop:numberMaxSubgroups}).

\subsection{Regular subsets and the function $X_T(h,k)$}
Let $T$ be an infinite abelian semigroup.
Let $A$ be an additive basis of $T$ such that 
$\ord_T^*(A) \le h$. 
What can be said about $\ord^*_T(A \setminus F)$ for those subsets $F \subset A$ such that $A \setminus F$ remains an additive basis of $T$? 
Such an $F \subset A$ is called a \em regular \em subset of $A$.

To tackle this problem, we define the function\footnote{In $\N$, this function is also denoted by $G_k(h)$ in the literature. 
Our notation accommodates the fact that we will be working with an infinite abelian group denoted by $G$, and also unifies different notations for the cases $k=1$ and $k > 1$.}
$$X_T(h,k) = \max_{\substack{A \subset T\\ hA \sim T}} \{ \ord_T^*(A \setminus F): F \subset A, F \textup{ is regular and } |F|=k\},$$
and set $X_T(h) = X_T(h,1)$.

In other words, $X_T(h,k)$ is the maximum order of a basis of $T$ obtained by removing a regular subset of cardinality $k$ from a basis of order at most $h$ of $T$. 

The function $X_{\N}(h)$ was  introduced by Erd\H{o}s and Graham in \cite{eg} under another name
and under this name by Grekos \cite{grekos0,grekos2}. It is known that
\begin{equation} \label{eq:x1}
 X_{\N}(h) \asymp h^2,
\end{equation}
see \cite{p2} for the best currently known bounds.
A conjecture of Erd\H{o}s and Graham \cite{eg2} asserting that $X_{\N}(h) \sim d_{\N}h^2$ for some absolute constant $d_{\N} >0$ as $h$ tends to infinity still stands to this day. 

The function $X_{\N}(h,k)$ was first introduced by Nathanson \cite{nathanson2}. For fixed $k \ge 1$ and $h \rightarrow \infty$, Nash and Nathanson \cite[Theorem 4]{nn} proved that 
\begin{equation} \label{eq:nn}
 X_{\N}(h,k) \asymp_{k}h^{k+1}.
\end{equation}
Their proof also yields $X_{\N}(h,k)\ll_h k^h$ for any fixed $h\ge 1$ and $k\rightarrow \infty$.
For a more detailed account and more precise estimates of $X_{\N}(h,k)$, we refer the reader to the survey \cite{jia}. 

In the context of infinite abelian groups, Lambert, Plagne and the third author \cite[Theorem 3]{llp} proved that, for a rather large class of infinite abelian groups $G$ (including $\Z^d$, any divisible group and the group $\Z_p$ of $p$-adic integers), one has 
\begin{equation} \label{eq:x2}
X_G(h) = O_G(h^2).
\end{equation}
However, the techniques do not carry over from these particular groups to arbitrary infinite abelian groups and, until now, it was not even known whether $X_G(h)$ is finite for all infinite abelian groups $G$ and integers $h \ge 1$. 
We now confirm that this is indeed the case and prove a Nash-Nathanson theorem in groups. 

\begin{theorem}\label{th:x2}
For any infinite abelian group $G$ and integer $k \ge 1$, we have
$X_G(h,k) \le \frac{h^{2k+1}}{k!^2} \left(1 + o_k(1)\right)$ as $h$ tends to infinity.
\end{theorem}
This bound may well not be optimal.
In fact, if $A$ is a basis of order $h$ of $G$ and $B\subset A$ a basis of cocardinality $k$, we find that
$B-B$ is a basis of order $O(h^{k+1})$, which is optimal.
In the regime where $h$ is fixed and $k$ tends to infinity, we find that  
$X_G(h,k)\le \frac{hk^{2h}}{h!^2} \left(1 + o_h(1)\right)$ holds.

Nash-Nathanson's proof of \eqref{eq:nn} uses Kneser's theorem\footnote{This is not the same as, but related to, Kneser's well known theorem on the cardinality of the sumset of two finite sets in an abelian group (see for instance \cite[Theorem 5.5]{tv}).} on the lower asymptotic density of sumsets in $\N$. 
Now such a theorem is not available in every infinite abelian semigroup $T$.
Our main tool in proving Theorem \ref{th:x2} will be \textit{invariant means}, that is, translation-invariant nonnegative functionals of norm 1 on $l^\infty(T)$ (see Section \ref{sec:invariant} for a precise definition). When restricted to indicator functions on $T$, an invariant mean gives rise to a ``density'' which is some sort of probability measure, albeit being only finitely additive. This notion of density is similar in many ways to the asymptotic density, but it is defined abstractly and it is less straightforward to infer properties of a set from its density. 
In \cite[Theorem 7]{llp}, invariant means were already used, but their use in the study of $X_T$ is new. 
We believe that invariant means will become part of the standard toolbox to study additive problems in abelian semigroups.

Imposing specific conditions on the semigroup $T$ allows one to control the function $X_{T}(h,k)$ more finely. 
We found a class of abelian groups for which a bound of the shape \eqref{eq:nn} may be achieved.
We say that a group $G$ is $\sigma$-\textit{finite}
if there exists a nondecreasing sequence $(G_n)_{n\in\N}$
of subgroups such that $G=\bigcup_{n\ge 0}G_n$.
Examples include $(C[x],+)$ for any finite abelian group $C$ or
$\bigcup_{n\ge 1}U_{d_n}$ where $U_k$ is the group of $k$-th roots of unity and
$(d_n)_{n\ge 1}$ is a sequence of integers satisfying $d_n\mid d_{n+1}$ for any $n\ge 1$;
the latter example includes the so-called Pr\"ufer $p$-groups 
$U_{p^\infty}$.
Combining a result of Hamidoune and R{\o}dseth \cite{hr} on this class of groups with the argument of Nash and Nathanson, we will prove the following bound.
\begin{theorem}
\label{th:sigmaFinite}
Let $G$ be an infinite $\sigma$-finite abelian group.
Then $X_G(h,k)\le 2\frac{h^{k+1}}{k!}+O(h^k)$.
\end{theorem}
In \cite[Theorem 5]{llp}, it was shown that for infinite abelian groups $G$ having a fixed exponent $p$, where $p$ is prime, $X_G(h)$ is in fact linear in $h$: $2h + O_p(1) \le X_G(h) \le ph + O_p(1)$. 
We now extend this to all infinite abelian groups having a prime power as an exponent, and show the same phenomenon for $X_G(h,k)$.
\begin{theorem} \label{th:x3}
Let $G$ be an infinite abelian group of finite exponent $\ell$. 
Then the following two statements hold.
\begin{enumerate}
\item $X_{G}(h,k) \le \ell^{2k} (h+1) - \ell^k +h $.
\item If $\ell$ is a prime power, then $X_G(h) \le \ell h + \ell^2 -\ell$.
\end{enumerate}
\end{theorem}
Finally, we discuss lower bounds.
Again, they depend on the finite quotients of the group.
In contrast to  the function $E_G$, and unsurprisingly in view of Theorem \ref{th:x3},
it is large cyclic quotients rather than large  quotients having small exponent which
cause $X_G$ to be large.
\begin{theorem}
\label{th:XlowBounds}
Whenever $G$ admits arbitrarily large cyclic quotients,
we have for each fixed $k$ and infinitely many $h$ the bound $X_G(h,k)\gg_k h^{k+1}$
and for each fixed $h$ and infinitely many $k$ the
other bound $X_G(h,k)\gg_h k^h$.
\end{theorem} 
Beyond $\Z$ and groups which admit $\Z$ as quotients, this property is satisfied by $\Z_p$ for 
any prime $p$ and $G=\bigcup_{n\ge 1}\prod_{m\le n}\Z/m\Z$, the latter being $\sigma$-finite. Combining with Theorem \ref{th:sigmaFinite} and equation \eqref{eq:x2}, 
we therefore have $X_G(h,k)\asymp_k h^{k+1}$
for this latter group and $X_{\Z_p}(h)\asymp X_{\Z^d}(h)\asymp h^2$.

\subsection{Semigroups}
The results announced so far reproduce in the frame of infinite abelian groups some
results known in the semigroup $\N$ (at least qualitatively). 
It turns out that our proofs do not entirely use the group structure, and are naturally valid in a wider class of semigroups which
comprises $\N$ and infinite abelian groups, which we term \textit{translatable} and which we will now describe.
Therefore, another aspect of our work is to unify the treatment of additive bases in $\N$ and
in abelian groups.
But since our results are new and interesting mostly in the case of groups, we decided to defer
the introduction of translatable semigroups until now.

An abelian semigroup $T$ is \textit{cancellative} if whenever $a, b, c \in T$ satisfy $a+c = b+c$, the relation $a=b$ holds. It is well known that such a semigroup is naturally embedded in a group $G_T$ called its Grothendieck group that
satisfies $G_T=T-T$ (see Section \ref{sec:translatable}).

A \emph{translatable} semigroup is an infinite cancellative abelian semigroup $(T,+)$ such that for any $x\in T$, the set $T\setminus (x+T)$ is finite; in other words, $T\sim x+T$.
Every infinite abelian group is a translatable semigroup.
Other examples of translatable semigroups include $\N$, numerical semigroups (i.e. cofinite subsemigroups of $\N$) and also $C \times \N$ for any finite abelian group $C$. These examples, in a sense, classify all translatable semigroups (see Proposition \ref{prop:structure}). In contrast, neither $(\N^d,+)$ for $d\ge 2$ nor $(\N^*,\times)$ are translatable.

Note that such a semigroup has the convenient property that whenever $A\subset T$ is a basis
and $x\in T$, then $x+A$ is still a basis, of the same order.
This property actually characterizes translatable semigroups, since $A=T$ is  a basis of order 1,
and bases of order 1 are precisely the cofinite subsets of $T$.

Having introduced translatable semigroups, we can now state our results in a more general setting.

\begin{proposition*}
Theorems \ref{th:e1},\ref{th:e2} and \ref{th:x2} are also valid when the group $G$ is replaced by a translatable semigroup $T$.
\end{proposition*}

The following identity says that the functions $E$ defined over $T$ and over its Grothendieck group are the same.
\begin{theorem}
\label{th:EGT}
Let $T$ be a translatable semigroup and $G$  its Grothendieck group.
Then for any $h,k\geq 1$ we have
$E_T(h,k)=E_G(h,k).$
\end{theorem}
In particular $E_\N=E_\Z$, which is already new. It also follows from Theorems \ref{th:EGT} and \ref{th:e2}$(iii)$ that $E_T(2,k) \le 2k-1$ for all translatable semigroups $T$. It is an interesting problem to determine if we also have $X_T(h,k)=X_G(h,k)$.

To conclude this subsection, observe that for general infinite abelian semigroups, even
cancellative ones, the finiteness of the set of essential subsets
(Theorem \ref{th:e1}) fails dramatically. Indeed, let $T=(\N^*,\times)$.
Let $A=\{2^k : k\in \N\}\cup\{2j+1 : j \in \N\}$.
Then the decomposition of any positive integer as a product of a power of 2 and an odd integer shows that $A$ is a basis of order 2. 
However, every prime is essential. 
Indeed, $h(A\setminus \{2\})$ does not meet $\{n \in \N : n\equiv 2\mod 4\}$ for any $h\ge 1$. 
If $p$ is an odd prime, the set $h(A\setminus \{p\})$ does not meet $\{2^kp: k\in \N\}$.
Hence, if one wants to keep this finiteness result, one needs to specify appropriate axioms. We stress that translatability is a joint generalization of $\N$ and infinite abelian groups. However, it may well be the case that the finiteness of $E_T(h,k)$ and $X_T(h,k)$ holds in an even more general class of semigroups. Indeed, one can show that $E_{\N^d}(h,k)$ is finite for any $h,k$ and $d$, though $(\N^d, +)$ is not translatable when $d \ge 2$. This is not obvious; in fact, even the fact that $\ord^*_{\N^d}$ is well-defined (i.e. if $h A \sim \N^d$ then $(h+1) A \sim \N^d$) is not obvious. We plan to investigate this finiteness phenomenon beyond translatable semigroups in the future.

\subsection{The "typical order" and the function $S_T(h,k)$}
Define $S_T(h)$ to be the minimum $s$ such that for any basis $A$ with $\ord_T^*(A) \le h$, there are only finitely many elements $a \in A$ such that $\ord_T^* (A \setminus \{a\} ) > s.$
In particular $S_T(h)\le X_T(h)$.
Grekos \cite{grekos2} introduced the function $S=S_\N$ and conjectured that $S_\N(h) < X_\N (h)$. 
In \cite{cp}, Cassaigne and Plagne settled Grekos' conjecture and proved that
\begin{equation*}
h+1 \le S_\N(h) \le 2h.
\end{equation*}

In \cite[Theorem 7]{llp}, using invariant means, it is shown that we also have $h+1 \le S_G(h) \le 2h$ for every infinite abelian group $G$. 
It is an open problem to find the exact asymptotic of $S_{\N}(h)$, or $S_{T}(h)$ for \em any \em fixed translatable semigroup $T$.
The proof of \cite[Theorem 7]{llp} also gives a bound for the number of "bad" elements, that is, elements $a$ of a basis $A$ of order at most $h$ such that
$S_G(h)<\ord_G^*(A\setminus \{a\})$. Further, it implies that the number of such elements is at most $h^2$. 
We now give a slight generalization of this fact to translatable semigroups, while showing that in the case of groups we do have a sharper bound.
\begin{theorem} \label{th:s1}
Let $T$ be a translatable semigroup, and let $h \ge 2$ be an integer. Then $S_{T}(h)\le 2h$.
In fact, if $A$ is a basis of $T$ of order at most $h$, then there are at most $h(h-1)$ elements $a$ of $A$ such that $\ord_T^* (A \setminus \{a\}) >2h $. 
If $T$ is a group then the number of such elements is at most $2(h-1)$.
\end{theorem}

While we do not know if $2(h-1)$ is best possible, it is nearly so because certainly $E_G(h)$ is a lower bound for the maximal number of bad elements, and it was observed in \cite[Theorem 2]{llp} that for the group $G=\F_2[t]$, one has $E_G(h) = h-1$ for any $h \ge 1$.

As a generalization, define $S_T(h,k)$ to be the minimum value of $s$ such that for any basis $A$ with $\ord_T^*(A) \le h$, there are only finitely many regular subsets $F \subset A, |F|=k$ with the property that $s < \ord_T^* (A \setminus F )$.
Thus $S_T(h,1) =S_T(h)$. We have the trivial bound $S_T(h,k) \le X_T(h,k)$, and it is interesting to know if this inequality is strict. 
We have a partial positive answer.
\begin{theorem} \label{th:s2}
Let $T$ be a translatable semigroup, and let $h \ge 1$ be an integer. 
Then 
\begin{equation} \label{eq:st}
S_T(h,2)\le 2X_T(h).
\end{equation}
Furthermore, if $A$ is a basis of $T$ of order at most $h$, there are at most $O \left( h^2 X_T(h)^2 \right)$ regular pairs $F \subset A$ such that $\ord_T^* (A \setminus F) > 2X_T(h)$. 
If $T$ is a group, then the number of such pairs is at most $4h (X_T(h)-1)$.
\end{theorem}
We underline that already in the semigroups $T=\N$ or $\Z$, the bound \eqref{eq:st} is nontrivial because $X_T(h)$ is much smaller than
$X_T(h,2)$.
Indeed,  $X_T(h,2)\gg h^3$ by \eqref{eq:nn} and Theorem \ref{th:XlowBounds} while $X_T(h)=O(h^2)$ by
\eqref{eq:x1} and \eqref{eq:x2}.
Thus $S_T(h,2)=O(X_T(h,2)/h)=o(X_T(h,2))$ as $h$ tends to infinity.

The organization of the paper is as follows. 
In Section \ref{sec:prelim} we introduce some tools used in our proofs, including a generalization of the Erd\H{o}s-Graham criterion for finite exceptional subsets. 
In Sections \ref{sec:e}, \ref{sec:x} and \ref{sec:s}, we prove results on the functions $E_T$, $X_T$, and $S_T$ respectively. 

\section{Preliminary results} \label{sec:prelim}
\subsection{Translatable semigroups and their Grothendieck groups} \label{sec:translatable}
Let $T$ be a cancellative abelian semigroup. 
We let $G_T$ be the quotient of the product semigroup $T\times T$ (with coordinate-wise addition) by the equivalence relation $\mathcal{R}$ defined by $(a_1,a_2)\mathcal{R} (b_1,b_2)$ if $a_1+b_2=a_2+b_1$. It is clear that the equivalence relation is compatible with the addition, so that the quotient is again an abelian semigroup. 
Further  the class of $(x,x)$ is a neutral element which we denote by 0 and $(a_1,a_2)+(a_2,a_1)=0$, so that $G_T$ is an abelian group.
Also $T$ is embedded in $G_T$ via the map $x\mapsto (x+t,t)$ (for any $t\in T$).
This group is called the \emph{Grothendieck group} of $T$.

By identifying $x \in T$ with $(x,0) \in G_T$, we have $T\subset G_T$, and we observe that $G_T=T-T$. 
We will often omit the index and let $G=G_T$.

Recall that a translatable semigroup is an infinite cancellative abelian semigroup with the property that for any $x\in T$, the set $T\setminus (x+T)$ is finite, or equivalently $T \sim x+T$.
We now list some immediate consequences of this property that we will use frequently.
\begin{lemma}
\label{lem:translatable}
Let $T$ be a translatable semigroup, $G=G_T$, and $H$ be a subgroup of finite index of $G$. Then
\begin{enumerate}
\item For any $x \in G$, we have $T \sim x+T$.
\item If $A$ is a subset of $G$, then for any $x\in G$, we have 
$T\cap (x+A) \sim x + T\cap A$.
\item If $F$ is a finite subset of $G$, then there is $t \in T$ such that $t+F \subset T$.
\item For any $x\in G$, $T\cap (x+H)$ is infinite.
\item $T\cap H$ is also a translatable semigroup. Furthermore, $H = T\cap H - T\cap H$.
\item If $R$ contains a system of representatives of $G/H$ and $S\subset G$ satisfies $T\cap H \subeq S$, then $T \subeq R+S$.
\end{enumerate}
\end{lemma}
\begin{proof}
Since $G=T-T$, we may write $x=a-b$ where $(a,b)\in T^2$.
Then
\[
x+T = (T+a)-b \sim T-b \sim (T+b) - b = T
\]
so that $x+T\sim T$ because the relation $\sim$ is transitive. 

If $A \subset G$, then
$$
T\cap (x+A)=x + (T-x) \cap A \sim x + T \cap A
$$
since $T \sim T-x$.

For part (3), for each $x \in F$ we write $x = a_x - b_x$ where $a_x, b_x \in T$. Thus $t = \sum_{x \in F} b_x$ satisfies $t \in T$ and $t+x \in T$ for all $x \in F$. 

If $H$ has finite index, there exists a finite set $F$ such that $G=\bigcup_{x\in F}(x+H)$.
By the pigeonhole principle, one of the sets $(x+H)\cap T$ for $x\in F$ must be infinite.
Hence all of them are infinite by part (2).

For part (5), the translatability of $T \cap H$ follows from part (2), since for any $x \in T\cap H$, we have $x + T \cap H \sim T \cap (x+H)=T\cap H$. 
Now let $x$ be any element of $H$. 
Then there exist $a , b \in T$ such that $x = a-b$. By part (4), there exists $c \in T$ such that $a+c \in H$. 
We also have $b+c \in H$. 
Therefore, $x = (a+c)-(b+c) \in T \cap H - T \cap H$. 

For part (6), notice that we may assume that $R$ is finite, and in this case,
$$ T = \bigcup_{r \in R} (r+H) \cap T \sim \bigcup_{r \in R} \left( r + H \cap T \right) \subeq \bigcup_{r \in R} (r + S) = R+S.$$
\end{proof}
We have a good understanding of the structure of translatable semigroups. Since we will not use this result, its proof is given in Appendix.
\begin{proposition*}[Proposition \ref{prop:structure}]
Let $T$ be a translatable semigroup. Then either $T$ is a group (i.e. $T$ equals its Grothendieck group $G_T$), or $T\sim C\oplus x\N$, where $x\in T$ and $C$ is a finite subgroup of $G_T$.
\end{proposition*}
 As a consequence of this structure result, any translatable semigroup $T$ admits a basis of any order $h\ge 2$ (Proposition \ref{prop:anyOrder}). This shows that our theorems are not vacuous in \textit{any} translatable semigroup.

In proving our results, we will often have to translate a basis by an element of $G_T$, and the translated set is not necessarily a subset of $T$. It is therefore advantageous to introduce a slightly more general notion of basis. 
For $A\subset G_T$, we say that $A$ is an \em additive $G_T$-basis \em (or simply a \em $G_T$-basis\em) of $T$ if there exists $h\ge 1$ such that $T\subeq hA$.
Again the order $\ord_{T}^*(A)$ of the basis $A$ over $T$ is then the minimal such $h$.
Note that any basis of $T$ is automatically also a $G_T$-basis of $T$ of the same order. We can then speak about regular, exceptional and essential subsets of $G_T$-bases in the same way as for bases.

\subsection{A generalization of the Erd\H{o}s-Graham criterion}
In the early eighties, Erd\H{o}s and Graham proved \cite[Theorem 1]{eg} that if $A$ is a basis of $\N$ and $a\in A$, then $A \setminus \{a\}$ is a basis of $\N$ if and only if $\gcd\left(A\setminus \{a\} - A\setminus \{a\}\right) =1$. 
This criterion was generalized to groups in \cite[Lemma 7]{llp}, as we now recall. 
Let $T$ be a translatable semigroup and $G_T$ be its Grothendieck group.
For $B\subset G_T$ (in particular for $B\subset T$), let $\langle B \rangle$ be the subgroup of $G=G_T$ generated by $B$.
The criterion states that if $A$ is a basis of $G$ and $a\in A$, then $A \setminus \{a\}$ is a basis of $G$ if and only if $ \langle A\setminus \{a\} - A\setminus \{a\} \rangle =G$. We now generalize further this criterion to translatable semigroups and exceptional subsets instead of exceptional elements. 

We first prove the following more general form of \cite[Lemma 7]{llp}.  
\begin{lemma} \label{lem:weakbasis}
Let $T$ be a translatable semigroup and $G$ be its Grothendieck group.
Let $s, t, h \ge 1$. Suppose $B \subset G$ and $a \in G$ satisfy
$$T \subeq   \bigcup_{i= h - t +1}^h (i B + (h-i)a ).$$
Suppose $(sB + a) \cap (s+1)B \neq \emptyset$ (in particular, this is the case if $sB-sB =G$). Then  $T \subeq h'B$ where $h'= (t-1) s + h$. 
\end{lemma}
\begin{proof}
Suppose $c \in (sB + a) \cap (s+1)B$. Then $2c \in (2sB +2a) \cap ((2s+1)B + a) \cap (2s+2)B$. Continuing in this way yields
$$(t-1) c \in \bigcap_{i=0}^{t-1}  \big( ((t-1) (s+1) -  i)B + ia \big).$$
For all but finitely many $x \in T$, we have $x \in (t-1)c + T$, and the hypothesis implies that for all but finitely many of them,
$$x  \in (t-1) c+ \bigcup_{i=0}^{t-1} ((i+h-t+1) B + (t-1-i)a).$$
It follows that for all but finitely many $x \in T$, we have
$$x = (x - (t-1) c) + (t-1) c \in ((t-1)(s+1) + h-t+1) B + (t-1)a = h' B + (t-1)a.$$
Since $T \sim T + (t-1)a$, this implies that $T \subeq h'B$, as desired.
\end{proof}

We can now state our generalization of the Erd\H{o}s-Graham criterion.
\begin{lemma} \label{lem:eg}
Let $T$ be a translatable semigroup and $G$ be its Grothendieck group.
Let $A$ be a $G$-basis of $T$. Let $F$ be a finite subset of $A$. 
Then $A \setminus F$ is a $G$-basis of $T$ if and only if $\langle A\setminus F - A\setminus F \rangle=G$. 
\end{lemma}

In the case of $\N$, this was proved by Nash and Nathanson in \cite[Theorem 3]{nn}. 
Their proof uses the fact that, in $\N$, any set of positive Schnirelmann density is a basis. 
Our argument is different from theirs.

\begin{proof}
Let $B=A\setminus F$. 
To prove the ``only if'' direction, let us suppose that $H = \langle B - B \rangle \subsetneq G$.
Let us prove that $T \not \subeq \ell B$ for any $\ell \ge 1$. 
Let $\ell \ge 1$. We may suppose that $\ell B \cap T$ is infinite, since otherwise we are done.
Note that $\ell B$ lies in a coset $x+H$ for some $x\in G$. 
In particular, $(x+H)\cap T$ is infinite. 
Let $y\in G\setminus (x+H)$; by Lemma \ref{lem:translatable} part (2), we have $(y+H)\cap T\sim y-x+(x+H)\cap T$ so $(y+H)\cap T$ is an infinite subset of $T$ that does not meet $\ell B$. In other words, $T \not \subeq \ell B$, as desired. 

We now prove the ``if'' direction. 
First, note that there exists $s\ge 1$ such that $sB\cap (s+1)B\neq \emptyset$. Indeed, let $b\in B$. 
Since $b\in G=\langle B-B\rangle$, there exists $s\ge 1$ such that $b\in s(B-B)$. 
Therefore, there exists $(x,y)\in (sB)^2$ such that $b=y-x$. 
Now $y=x+b\in sB\cap (s+1)B$ yields the desired nonempty intersection.
According to Lemma \ref{lem:weakbasis} (with $a=0$), it now suffices to show that $T \subeq \bigcup_{i=1}^{\ell} iB$ for some $\ell \ge 1$. 
Since $\langle B - B \rangle=G$, each element $x \in F$ has a representation of the form
 \begin{equation} \label{eq:repx}
  x = \sum_{i=1}^{s_x} (a_i(x) - b_i(x)),
 \end{equation}
where $s_x\in\N$ and $a_i(x), b_i(x) \in B$.
Since $A$ is a $G$-basis of $T$, let $h\ge 1$ satisfy $T \subeq hA$. 
All but finitely many elements $g \in T$ can be written as
$$g = \sum_{x \in F} m_x(g) x +y,$$
where $m_x(g) \ge 0$ and $\sum_{x \in F} m_x(g) \le h$ whereas $y\in(h-\sum_{x\in F}m_x(g))B$. 
Replacing each occurrence of $x \in F$ with \eqref{eq:repx} and translating by $g_0=h \sum_{x \in F} \sum_{i=1}^{s_x} b_i(x)\in T$, we find that 
\begin{align*}
g + g_0 = & \sum_{x \in F} \sum_{i=1}^{s_x} \left( m_x(g) a_i(x) + (h-m_x(g)) b_i(x) \right) + y,
\end{align*}
where the right-hand side is a sum of 
$$h \sum_{x \in F} s_x + h - \sum_{x \in F} m_x(g)$$ 
elements in $B$. 
Let $\ell = h \sum_{x \in F} s_x + h$.
This shows that $g_0+T \subeq \bigcup_{i=1}^{\ell} iB$ and by translatability, we conclude.
\end{proof}
As pointed out by Nash and Nathanson \cite{nn}, the conclusion of Lemma \ref{lem:eg} is no longer true for the semigroups $T=\N$ or $T=\Z$ if $F \subset A$ is allowed to be infinite. For example, consider $A = \{1 \} \cup \{2n: n \in T \}$, a basis of order 2 of $T$, and $F=\{n\in T : \forall k\ge 1, \, n \neq 6^k \}$. 

More generally, let $T$ be a translatable semigroup and $h\ge 2$. 
We invoke the construction of a basis $A$ of order $h$ in
Proposition \ref{prop:anyOrder}.
With the notation of that construction, let $B=\bigcup_{i=0}^\infty \Lambda_i\subset A$ and $F=A\setminus B$.
Then $\langle B-B\rangle =G_T$.
However, for any $\ell \ge 1$, the sumset $\ell B$ misses all elements whose support has cardinality strictly larger than $\ell$, so $B$ is not a basis.
This means that in any translatable semigroup, the finiteness of $F$ is crucial for Lemma \ref{lem:eg}.

\subsection{Characterizations of exceptional and essential subsets}
As demonstrated by Lemma \ref{lem:eg}, the subgroups $\langle A\setminus F - A\setminus F \rangle$, where $F$ is a finite subset of a given basis $A$, play an important role. We now prove some preliminary results on these subgroups. 
The next lemma states that whenever $F$ is a finite subset of $A$, the subgroup $\langle A\setminus F - A\setminus F \rangle$ cannot be too small.
\begin{lemma} \label{lem:index}
Let $T$ be a translatable semigroup and $G$ its Grothendieck group. 
Let $A$ be a subset of $G$ such that $T \subeq hA$ for some $h\ge 2$ and let $F$ be a finite subset of $A$. 
Let $H=\langle A\setminus F - A\setminus F \rangle$.
Then for any $x\in A\setminus F$, we have $(h-1)(F\cup\{x\})+H=G$.
Consequently,
$$[G:H] \le \binom{h + |F|-1}{h-1}.$$
\end{lemma}
\begin{proof}
By the definition of $H$, we have $A\setminus F \subset x +H$, so that $A\subset (x+H)\cup F$ and $A$ meets a finite number of cosets of $H$.
This fact and the finiteness of $T\setminus hA$ imply that the projection of $T$ in $G/H$ is finite. 
However, $T-T=G$, so $G/H$ is finite.

Let $g\in G/H$. We may write $g=t+H$ for some $t\in T$. 
Now $hA\subset hF\cup \bigcup_{i=0}^{h-1}(iF+(h-i)x+H)$.
Note that $hF\cup (T\setminus hA)$ is finite and $(t+H)\cap T$ is infinite by Lemma \ref{lem:translatable}. 
This implies that $g=t'+H$ for some $t'\in \bigcup_{i=0}^{h-1}(iF+(h-i)x)$.
Finally,
\[
G\subset H+\bigcup_{i=0}^{h-1}(iF+(h-i)x)=H+x+
(h-1)(F\cup \{x\})
\]
as desired. 
This implies that $[G : H]\le \abs{(h-1)F'}$ where $F'=F\cup \{x\}$ has cardinality $\abs{F}+1$.
The bound follows from counting the number of $(h-1)$-combinations of elements from $F'$ with repetition allowed.
\end{proof}

We are now ready to prove the first part of Theorem \ref{th:ElowBounds}.

\begin{proof}[Proof of Theorem \ref{th:ElowBounds}(i)]
Lemma \ref{lem:index} implies that if $G$ does not have proper subgroups of index at most $\binom{h + k-1}{h-1}$, then a basis $A$ of order at most $h$ cannot contain an exceptional (and in particular essential) subset $F$ of cardinality at most $k$. 
This yields the first implication of Theorem \ref{th:ElowBounds}$(i)$.
For the second one, let $G$ be an infinite abelian group and $H$ a proper subgroup of finite index. 
Let $R$ be a (finite) set of distinct representatives modulo $H$. Then $A=R\cup H$ is a basis of order 2
of $G$ and $R$ is an exceptional set, which contains an essential subset. Therefore, $E_G(2,k) >0$ for some $k$. (We will encounter similar arguments in Section
\ref{sec:lowbounds}.)
\end{proof}

Lemma \ref{lem:eg} gives the following characterization of essential subsets of a basis. 
\begin{corollary} \label{cor:ess}
Let $T$ be a translatable semigroup, $G$ its Grothendieck group and $A$ a $G$-basis of $T$ and $E \subset A$ be a finite subset. 
Then $E$ is an essential subset of $A$ if and only if the following two statements hold.
\begin{enumerate}
\item $H=\langle A\setminus E - A\setminus E \rangle$ is a proper subgroup of $G$.
\item $G/H$ is generated by $\overline{x-a}$, where $x$ is any element of $E$ and $a$ is any element of $A \setminus E$.
\end{enumerate}
In particular, if $E$ is essential then $G/H$ is a finite cyclic group.
\end{corollary}
\begin{proof}
Lemma \ref{lem:eg} implies that $E$ is essential precisely when $G \neq H$, but $G = \langle (A\setminus E)\cup \{x \} - (A\setminus E) \cup \{x \} \rangle$ for any $x \in E$. 
The claimed characterization follows by noting that $\langle (A\setminus E)\cup \{x \} - (A\setminus E) \cup \{x \} \rangle$ is generated by $H \cup \{ x-a \}$ for any $a \in A\setminus E$. 

The second claim follows from the fact that $G/H$ is finite, by Lemma \ref{lem:index}. 
\end{proof}

The next lemma gives a correspondence between  essential subsets and proper subgroups.
\begin{lemma}
\label{lem:strictSubgroup}
Let $T$ be a translatable semigroup, $G$ its Grothendieck group  and $A$ be a $G$-basis of $T$.
Let $E$ be an  essential subset of the basis $A$ and $F$ be any  subset of $A$ such that $E\not \subset F$.
Then $\langle A\setminus (E\cup F) - A\setminus (E\cup F) \rangle \subsetneq \langle A\setminus F - A\setminus F \rangle$.
\end{lemma}
\begin{proof}
We have
$\langle A\setminus (E\cup F) - A\setminus (E\cup F)\rangle
\subset \langle A\setminus E - A\setminus E\rangle \cap
\langle A\setminus F - A\setminus F \rangle$.
Further, since $ A\setminus (E\cap F)=(A\setminus E) \cup (A\setminus F)$, we have
\[
\langle A\setminus E - A\setminus E\rangle +
\langle A\setminus F - A\setminus F \rangle=\langle A\setminus (E\cap F) - A\setminus (E\cap F)\rangle.
\]
Since $E\cap F\subsetneq E$, it follows from the essentiality of $E$ and Lemma \ref{lem:eg} that the right-hand side is $G\neq \langle A\setminus E - A\setminus E\rangle$.
So $\langle A\setminus F - A\setminus F\rangle \not \subset \langle A\setminus E - A\setminus E \rangle$, which finally yields the desired result.
\end{proof}

\subsection{Invariant means} \label{sec:invariant}
Let $(T,+)$ be an abelian semigroup. Let $\ell^{\infty}(T)$ denote the set of all bounded functions from $T$ to $\R$. 
An \em invariant mean \em on $T$ is a linear functional $\Lambda: \ell^{\infty}(T) \rightarrow \R$ satisfying the following conditions.
\begin{enumerate}
\item[(M1)] $\Lambda$ is \textit{nonnegative}: if $f \ge 0$ on $T$, then $\Lambda(f) \ge 0$.
\item[(M2)] $\Lambda$ has \textit{norm} 1: $\Lambda(1_T)=1$ where $1_T$ is the characteristic function of $T$.
\item[(M3)] $\Lambda$ is \textit{translation-invariant}: $\Lambda( \tau_x f) = \Lambda( f)$ for any $f \in \ell^{\infty}(T)$ and $x \in T$, where $\tau_x$ is the translation by $x$: $\tau_x f(t) = f(x+t)$.
\end{enumerate}

Note that by restricting $\Lambda$ to indicator functions of subsets of $T$, we induce a function $d: \mathcal{P} (T) \rightarrow [0,1]$, that we will usually call \emph{density} satisfying the following three properties.
\begin{enumerate}
\item[(D1)] $d$ is \em finitely additive\em, i.e. if $A_1, \ldots, A_n \subset T$ are disjoint, then
$$d \left(\displaystyle\bigcup_{i=1}^n A_i\right) = \sum_{i=1}^n d(A_i).$$
\item[(D2)] $d$ is \em translation-invariant\em, i.e. for all $A \subset T$ and $x \in T$, we have $d(x+A) = d(A)$.
\item[(D3)] $d$ has  \em total mass \em $1$, i.e. $d(T) =1$.
\end{enumerate}
Note that the axiom (D1) implies that for any $A_1, \ldots, A_n \subset T$, we have $d(\bigcup_{i=1}^n A_i)\le \sum_{i=1}^n d(A_i)$. 
Also, if $A$ is finite, then $d(A)=0$.

If there exists an invariant mean on $T$, then $T$ is said to be \textit{amenable}. 
It is known that all abelian semigroups are amenable (for a proof, see \cite[Theorem 6.2.12]{landmark}). 
However, even in $\N$, all known proofs of the existence of invariant means are nonconstructive\footnote{Observe that popular densities such as the lower asymptotic one do not satisfy the equality of (D1): only an inequality is true in general.}, and require the axiom of choice in one way or another (e.g. the Hahn-Banach theorem or ultrafilters). 

In Sections \ref{sec:x} and \ref{sec:s}, we will use the existence of invariant means as a blackbox and make crucial use of their properties to prove our results. 
For now, we record the following simple fact, which is an immediate extension of the so-called prehistorical lemma to invariant means.
\begin{lemma}\label{lem:prehistoric}
Let $T$ be a cancellative abelian semigroup, $G$ be its Grothendieck group and $d$ be a density on $T$. 
If $A, B \subset T$ and $d(A) + d(B) > 1$ then $T \subset A-B\subset G$. In particular, if $T$ is a group then $T=A-B$.
\end{lemma}
\begin{proof}
Let $t\in T$. By (D2), $d(A)+d(t+B)=d(A)+d(B)>1$.
By axioms (D1) and (D3), we infer that $A\cap (t+B)\neq\emptyset$.
Let $a=t+b\in A\cap (t+B)$, then $t=a-b\in A-B$.
\end{proof}

We will also make use of the following observation, which says that if $T$ is translatable, then any invariant mean on $T$ can be extended to all of $G$ in a trivial way. 
\begin{lemma} \label{lem:extension}
Let $T$ be a translatable semigroup, $G$ be its Grothendieck group and $\Lambda$ be an invariant mean on $T$. 
For $f \in \ell^{\infty}(G)$, define $\Lambda' (f) = \Lambda(f |_T)$, where $f |_T$ is the restriction of $f$ on $T$. 
Then $\Lambda'$ is an invariant mean on $G$. 
\end{lemma}
\begin{proof}
Since $G=T-T$, it suffices to verify (M3) for any $f \in \ell^{\infty}(G)$ and $x \in T$. We have 
\begin{eqnarray*}
\Lambda'( \tau_x f) &=& \Lambda( (\tau_x f)|_T) = \Lambda( \tau_x (f|_{T-x}) ) \\
&=& \Lambda( \tau_x (f|_{T} ) ) + \Lambda( \tau_x (f|_{(T-x)\setminus T} ) ) \\
&=& \Lambda( f|_{T}  ) +  \Lambda( f|_{T\setminus (T+x)} )  \\
&=& \Lambda'( f  ) 
\end{eqnarray*}
since $T\setminus (T+x)$ is finite and $f$ is bounded.
\end{proof}

When $T$ is a group, in proving Theorem \ref{th:s1}, we will require the following additional property of $d$.
\begin{enumerate}
\item[(D4)] $d$ is invariant with respect to inversion, i.e. $d(A) = d(-A)$ for all $A \subset T$. 
\end{enumerate}
This property may not be satisfied by all invariant means, but invariant means having this property abound (see for instance \cite[Theorem 1]{chou}).

\section{Essential subsets of an additive basis} \label{sec:e}

\subsection{Finiteness of the set of essential subsets} We first prove Theorem \ref{th:e1}.

\begin{proof}[Proof of Theorem \ref{th:e1}]
Let $T$ be a translatable semigroup and $G$ be its Grothendieck group. 
Let also $A$ be an additive $G$-basis of order $h \ge 1$ over $T$. 
To obtain a contradiction, we assume that the set $\mathcal{F}_A$ of all  essential subsets of $A$ is infinite.
It follows that $h \ge 2$ and there exists an infinite sequence $(F_i)_{i \ge 1}$ of pairwise distinct elements of $\mathcal{F}_A$. 
In addition, extracting an appropriate infinite subsequence of $(F_i)_{i \ge 1}$ if need be, we may assume that $F_{i+1}\not \subset\bigcup_{j=1}^i F_j$ for all $i \ge 1$.

Let us set $H_i=\langle A\setminus \bigcup_{j=1}^iF_j-A\setminus \bigcup_{j=1}^iF_j\rangle$ for all $i \ge 1$.
On the one hand, it follows from Lemma \ref{lem:strictSubgroup} that $(H_i)_{i \ge 1}$ is a decreasing sequence of proper subgroups of $G$, and from Lemma \ref{lem:index} that, for every $i \ge 1$, the quotient group $G_i=G \slash H_i$ is finite (in particular, $H_i$ is infinite).
On the other hand, for every $i \ge 1$, there is a unique coset $K_i$ of $H_i$ such that $A\setminus K_i$ is finite.
In particular, one has $K_j \subset K_i$ for any $j \ge i$.

Now, for each $i \ge 1$, let us define $d_i=\min\{\ell \ge 1 : |(\ell A )\cap X| = \infty, \ \forall X \in G_i\}$.
In other words, $d_i$ is the smallest integer $\ell \ge 1$ such that every coset of $H_i$ has an infinite intersection with $\ell A$.  
Alternatively,  one also has $d_i=\min\{\ell \ge 1 : G_i \subset K_i + (\ell-1)\pi_i(A)\}$ where, for every $i \ge 1$, $\pi_i$ denotes the canonical epimorphism from $G$ to $G_i$.

It is easily noticed that by definition, the sequence $(d_i)_{i \ge 1}$ is nondecreasing. 
Also, since $H_i$ is a proper subgroup of $G$ and $K_i$ is the only coset of $H_i $ having an infinite intersection with $A$, one has $2 \le d_i $ for all $i$. 
Finally, since $T\setminus hA$ is finite by assumption and each coset of $H_i$ has an infinite intersection with $T$, one has $d_i\le h$ for all $i$. 

At this stage, observe that by translatability, any translation of the original additive $G$-basis $A$ by an element $a \in G$ results in a new additive $G$-basis $A'=a+A$ of order $h$ itself over $T$. 
The sequence $(F'_i)_{i \ge 1}$ obtained by translating each $F_i$ by $a$ is then an infinite sequence of essential subsets of $A'$ satisfying $F'_{i+1}\not \subset\bigcup_{j=1}^i F'_j$ for all $i \ge 1$, and starting from which the previous definitions yield the very same sequences $(H_i)_{i \ge 1}$ and $(d_i)_{i \ge 1}$ as for $A$ itself.

Since the sequence $(d_i)_{i\ge 1}$ is nondecreasing and bounded,
we infer that it is stationary. Thus, we may fix some
$i\ge 1$ so that $d_j=d_i$ for all $j\ge i$.
Let $x_i \in G$ be such that $K_i=x_i+H_i$.
Now, using the just described translation-invariance of $(H_i)_{i \ge 1}$ and $(d_i)_{i \ge 1}$, we can assume from now on that $x_i=0$ and $K_i=H_i$. 
In particular, $d_i=\min\{\ell \ge 1 : G_i \subset (\ell-1)\pi_i(A)\}$.

It follows from the minimality of $d_i \ge 2$ that $G_i \subset (d_i-1)\pi_i(A)$ while $G_i \not \subset (d_i-2)\pi_i(A)$. In other words, there exists at least one coset $C_i$ of $H_i$ belonging to $(d_i-1)\pi_i(A) \setminus (d_i-2)\pi_i(A)$.
Now, pick an integer $j \ge i$. Recall that $d_j=d_i$ and let $L$ be any coset of $H_j$ such that $L \subset C_i$.

Since $d_j=d_i$, one has $L \in K_j + (d_j-1)\pi_j(A)=K_j + (d_i-1)\pi_j(A)$. 
Let $L_2,\dots,L_{d_i}$ be any $d_i-1$ elements of $\pi_j(A)$ such that in $G_j$, one has
$$L=K_j+L_2+ \dots + L_{d_i}.$$
For every $j \ge i$, let $f^i_j : G_j \rightarrow G_i$ be the group homomorphism sending every coset $L$ of $H_j$ to the unique coset of $H_i$ containing $L$.
Note also that by definition, one has $f^i_j \circ \pi_j = \pi_i$.
Since $f^i_j(K_j)=K_i =H_i$, applying $f^i_j$ to both sides of the  equality above in $G_j$ results in the following relation in $G_i$,
$$C_i=f^i_j(L_2)+ \dots + f^i_j(L_{d_i}).$$
For every $2 \le k \le d_i$, there exists by definition an element $a_k \in A$ such that $L_k=\pi_j(a_k)$.
However, $a_k \in K_i$ would imply $f^i_j(L_k)=(f^i_j \circ \pi_j)(a_k)=\pi_i(a_k)=K_i=H_i$ and readily give $C_i \in (d_i-2)\pi_i(A)$, which is a contradiction. 
As a result, $L_k \in \pi_j(A \setminus K_i)$, for every $2 \le k \le d_i$.
We now have all we need to complete our proof.

On the one hand, each $L_k$ can take at most $|\pi_j(A \setminus K_i)| \le |A \setminus K_i|$ values, so that the number of possible sums of the form $L_2+ \dots + L_{d_i}$ in $G_j$ is at most 
\[
{|A \setminus K_i|+d_i-1 \choose d_i},
\] 
which is independent of $j$.

On the other hand, there are $[H_i : H_j]$ cosets of $H_j$ that are contained in $C_i$, and in order for each of them to be an element of $K_j + (d_i-1)\pi_j(A)$, we must have 
$$[H_i : H_j] \le {|A \setminus K_i|+d_i-1 \choose d_i}.$$ 
However, since $[H_i : H_j]$ tends to infinity when $j$ does so, the previous inequality can only hold for finitely many integers $j \ge i$.
Therefore, the initial assumption that $\mathcal{F}_A$ is infinite leads to a contradiction indeed.
\end{proof}

\subsection{Bounding the number of essential subsets} We now prove parts $(i)$ and $(ii)$ of Theorem \ref{th:e2}.
Let $T$ be a translatable semigroup and $G=G_T$ be its Grothendieck group.
Let $h\ge 2,k\ge 1$ be integers.
The case $k=1$ is already covered in \cite[Theorem 2]{llp}
when $T$ is a group, and by Theorem \ref{th:EGT}, which we prove in the next section without using the current section, we 
conlude that $E_T(h,1)=E_G(h,1)\leq h-1$.
Henceforth we shall assume $k\ge 2$.
Let $A$ be a $G_T$-basis of order at most $h$ over $T$.
It readily follows from Theorem \ref{th:e1} that the set $\mathcal{F}$ of essential subsets of cardinality $k$ of $A$ is finite. 
Our aim is to bound $N=|\mathcal{F}|$ in terms of $h$ and $k$ alone. We will actually prove the following precise bounds:
\begin{equation} \label{eq:N1}
N \le (30 h \log k)^k
\end{equation}
and
\begin{equation} \label{eq:N2}
N \le (32 k \log k)^{h-1}.
\end{equation}
We first prove the following lemma.
\begin{lemma}
\label{lem:nBound}
Take a minimal sequence $F_1,\ldots, F_{n}$ of elements of $\mathcal{F}$ with the property that $\bigcup_{i\le n} F_i=\bigcup_{F\in\mathcal{F}}F$.
Then 
$n\le (h-1)(\log_2(ek\log_2 k)+\epsilon_k)$,
where $\epsilon_k\in [0,5]$ tends to zero as $k$ tends to infinity
and does not depend on $h$.
In particular $n\leq 11(h-1)\log k$.
\end{lemma} 
Thus $n\leq (h-1)(1+o_{k\rightarrow\infty}(1))\log_2  k$.
We will show later (Remark \ref{r:optimal}) that this bound is sharp, up to the error term. This is why we decided to state
it as an independent lemma.
\begin{proof}
By minimality, note that $F_i\not \subset\bigcup_{1\le j< i}F_j$ for any $i\in [n]$. 
Also, $|\bigcup_{F\in\mathcal{F}}F| \le nk$. 
Let $H_{i} =\langle A \setminus \bigcup_{j=1}^i F_j - A \setminus \bigcup_{j=1}^i F_j \rangle$. 
By Lemma \ref{lem:strictSubgroup}, one has
$$H_{n} \subsetneq H_{n-1} \subsetneq \cdots \subsetneq H_1 \subsetneq G.$$ 
Therefore,
\begin{equation} \label{eq:index-upper}
[G:H_{n}] \ge 2^n.
\end{equation}
On the other hand, it follows from Lemma \ref{lem:index} that
\begin{equation} \label{eq:index-lower}
[G:H_{n}] \le \binom{h + nk -1}{h-1}
\end{equation}
since $|\bigcup_{j=1}^n F_j| \le kn$. 
Combining \eqref{eq:index-upper} and \eqref{eq:index-lower}, one has $2^{n} \le \binom{h + nk -1}{h-1}$.
Using the elementary bound ${a \choose b} \le (\frac{ea}{b})^b$, we have 
\[
2^n \le \left( \frac{e(h-1+nk)}{h-1} \right)^{h-1}. 
\]
Denoting $t=\frac{n}{h-1}$, 
this implies that 
$ t\le \log_2 (e(1+kt))$.
Consider the real function $f:x\mapsto x-\log_2(e(1+ x k))$.
Observe that 
\begin{align*}
f(\log_2(ek\log_2k))&=\log_2(ek\log_2k)-\log_2(e(1+k\log_2(ek\log_2k))\\
&=\log_2(ek\log_2k)-\log_2 \left( ek\log_2k \left( 1+\frac{1}{k\log_2k}+\frac{\log_2(e\log_2k)}{\log_2k} \right) \right)\\
&=-\log_2 \left( 1+\frac{1}{k\log_2k}+\frac{\log_2(e\log_2k)}{\log_2k} \right)\\
&\ge -\frac{1}{k\log_2k}-\frac{\log_2(e\log_2k)}{\log_2k}\\
&> -1.35\frac{\log_2(e\log_2k)}{\log_2k}.
\end{align*}
Further, $f'(x)\ge 1-1/(x\log 2)$, in particular
$f'(x)\ge 1-1/\log(2e)\ge 2/5$ for all $x\ge \log_2(ek\log_2k)\ge \log_2(2e)$ (here we use $k\ge 2$).
Hence $f(x)>0$ whenever $x>\log_2(ek\log_2k)+3.4\frac{\log_2(e\log_2k)}{\log_2k}$.
Since $f(t)\le 0$, this yields the desired bound, with $\epsilon_k=3.4\frac{\log_2(e\log_2k)}{\log_2k}\le 5$.

To obtain $n\leq 11(h-1)\log k$,
it suffices to note that $e\log_2k\leq k^{3/2}$ for all $k\ge 2$
and $\epsilon_k\le 5$ so that 
$\log_2(ek\log_2 k)+\epsilon_k\le 7.5\log_2k\le 11\log k$.
\end{proof}
We return to the proof of Theorem \ref{th:e2}.
We start with the first item of that theorem.
Using again the bound ${a \choose b} \le (\frac{ea}{b})^b$, we have 
\begin{equation} \label{eq:N}
N\le \binom{nk}{k}\le (en)^k.
\end{equation}
Injecting Lemma \ref{lem:nBound} above yields $N \le (30 h \log k)^k\ll_k h^k$, as desired.\\

We now prove the second item of Theorem \ref{th:e2}. Let
$H =\langle A \setminus \bigcup_{F\in\mathcal{F}}F - A \setminus \bigcup_{F\in\mathcal{F}}F \rangle$.
Then we prove the following lemma.
\begin{lemma}
There exists an injection from $\mathcal{F}$ to the set of cyclic subgroups of $G/H$.
\end{lemma}
\begin{proof}
First consider the map $E\mapsto H_E=\langle A\setminus E-A\setminus E\rangle$ defined on $\mathcal{F}$.
It is injective because of Lemma \ref{lem:strictSubgroup}.
Note that $H\le H_E$ for any $E\in\mathcal{F}$.
Further the map $H_E\mapsto H_E/H$ is also an injection (as a restriction of the classical bijection
between subgroups of $G$ containing $H$ and subgroups of $G/H$).
Let $Q=G/H$. This is a finite abelian group. The theory of characters of finite
abelian groups implies that there exists an involution $f$ of the set of subgroups of $Q$ such that
for any $K\le Q$, the groups $K$ and $ Q/f(K)$
are isomorphic; cf \cite{groupprops}. Consider the map $E\mapsto f(H_E/H)$, which is injective as a composition
of three injective maps.
Finally, $f(H_E/H)\simeq (G/H)/(H_E/H)\simeq G/H_E$, which is cyclic by Corollary \ref{cor:ess}. 
\end{proof}
In particular, we have $N\le [G:H]$.
Because of Lemma \ref{lem:index}, we infer $N\le \binom{h-1+nk}{h-1}\le \left( \frac{e(h-1 + nk)}{h-1}\right)^{h-1}$
and injecting again Lemma \ref{lem:nBound}, we conclude $N \le \left( e (1  + 11 k \log k) \right)^{h-1} \le (32 k \log k)^{h-1}$.
\begin{remark}
Note that these proofs also show that $E_T(h,\le k) \ll_k h^k$ and $E_T(h,\le k) \ll_h (k \log k)^{h-1}$. One must simply replace \eqref{eq:N} by 
$N\le k\binom{nk}{k}$ by unimodularity of the binomial coefficients and the assumption that $n\ge 2$ (otherwise $n=N=1$).
\end{remark}

\begin{remark} \label{r:optimal}
We point out that Lemma \ref{lem:nBound}
is optimal up to the second order terms, for any $h$ and $k=2^{r-1}$ for some $r \ge 1$. Indeed, let $G=\bigoplus_{i=1}^h G_i$, where $G_i \cong \F_2^r$ for $i=1, \ldots, h-1$ and $G_h$ is an infinite group. Then 
$A = \bigcup_{i=1}^h G_i$ is a basis of order $h$ of $G$. Using Lemma \ref{lem:eg}, one can show that $F\subset A$ is an essential subset if, and only if, for some $ i \in [h-1]$, $F \subset G_i$ and $G_i\setminus F$ is a maximal subgroup of $G_i$. In other words, $F$ is the complement of a hyperplane of $G_i$ and has cardinality $k=2^{r-1}$. For each $1 \le i \le h-1$, in order to cover all complements of hyperplanes of $G_i$, we need at least $r$ of them. Hence, in order to satisfy the hypothesis of Lemma \ref{lem:nBound}, we need $n \ge (h-1)r = (h-1)(\log_2k+1)$.
What may not be optimal in the bounds \eqref{eq:N1} and \eqref{eq:N2} is how we infer an upper bound for the total number $N$ of essential subsets
from the upper bound on $n$.
\end{remark}

\subsection{Comparing $E_T$ and $E_{G_T}$} In this section
we prove Theorem \ref{th:EGT}. We first need the following generalization of Lemma \ref{lem:eg}.
\begin{lemma} \label{lem:nn}
Let $T$ be a translatable semigroup and $G$ its Grothendieck group. Let $A$ be a $G$-basis of $T$ and $F \subset A$ be any finite subset. 
Put $B=A\setminus F$ and $H= \langle B - B \rangle$. 
Let $b$ be an arbitrary element of $B$.
Then $T\cap H$ is a translatable semigroup of Grothendieck group $H$ and $B-b$ is an $H$-basis of $T\cap H$.
\end{lemma}

Clearly, Lemma \ref{lem:eg} is a special case of Lemma \ref{lem:nn} when $H=G$. In $\N$, Lemma \ref{lem:nn} was proved by Nash-Nathanson \cite[Theorem 1]{nn}. Again, Nash-Nathanson's proof is very specific to $\N$ (it uses Schnirelmann density and Schnirelmann's theorem). 
Our proof is different from theirs and works for any translatable semigroup. 
In fact, we use Lemma \ref{lem:eg} to prove Lemma \ref{lem:nn}, while Nash and Nathanson proceeded the other way round.

\begin{proof}
The fact that $T\cap H$ is a translatable semigroup of Grothendieck group $H$ is Lemma \ref{lem:translatable} part (5).
For $h$ large enough, and by translatability, we have
\begin{eqnarray*}
  T\sim T-hb \subeq h(A-b) &=& \bigcup_{i=0}^h \left(i(F-b) + (h-i)(B -b) \right) \\
 & \subset & \bigcup_{i=0}^h \left(i(F-b) + h(B -b) \right) \qquad \textup{since $0 \in B-b$}.
 \end{eqnarray*}
In particular,
\[
T \cap H \subeq \bigcup_{i=0}^h \left(i(F-b) + h(B-b) \right).
\] 
Since $F$ is finite, this means that there are finitely many translates $a_1 + h(B-b),\ldots, a_k + h(B-b)$ of $h(B-b)$ such that 
\[
T \cap  H \subeq \bigcup_{i=1}^k \left(a_i + h(B-b)\right).
\]
A priori $a_1, \ldots, a_k \in G$. But a translate $a_i+h(B-b)$ can have nonempty intersection with $H$ only if $a_i \in H$. Thus we may assume that $a_1, \ldots, a_k \in H$. 
Let $A' = h(B-b) \cup \{a_1, \ldots, a_k\} \subset H$, then the equation above shows that $T\cap H \subeq 2 A'$. Clearly $\langle hB - hB \rangle = H$. We now invoke Lemma \ref{lem:eg} with the set $A'$ and the translatable semigroup $T\cap H$ (whose Grothendieck group is $H$), and conclude that for some $k \ge 1$, 
$T\cap H \subeq kh(B-b)$, as desired. 
\end{proof}

Next we need the following lemma of independent interest, which is reminiscent of Hegarty's reduction \cite{hegarty2} of the study of $E_\N(h,k)$ to the postage stamp problem.
\begin{lemma}
\label{lem:twobases}
Let $T$ be a translatable semigroup 
and $G$ its Grothendieck group.
Let $H$ be a subgroup of $G$ of finite index.
Let $B$ be a subset of $G$ satisfying $\langle B  - B\rangle=H$ and $b$ be an arbitrary element of $B$. 
Let $F$ be a finite subset of $G$ disjoint from $B$ and $A=F\cup B$. 
Then the following assertions are equivalent:
\begin{enumerate}
\item $A$ is a $G$-basis of $T$.
\item 
\begin{enumerate}
\item $B-b$ is an $H$-basis of   $T \cap H$, and 
\item $\langle F-b + H\rangle =G$ (i.e. $\overline{F-b}$ generates $G/H$).
\end{enumerate}
\end{enumerate}
Further, if $h_1$ is minimal such that $h_1 ( (F-b) \cup \{0\})+H=G$, $h_2=\ord^*_{T\cap H} (B-b)$, and $h=\ord^*_T (A)$, then we have $h_1+1\le h \le h_1+h_2$. 
\end{lemma}
\begin{proof}
If (1) holds, then (2a) follows from Lemma \ref{lem:nn} and (2b) follows from Lemma \ref{lem:index}.

Now suppose (2) holds. Let $h_1$ be minimal such that $h_1 ( (F-b) \cup \{ 0 \})+H=G$ and $h_2=\ord^*_{T\cap H} (B-b)$. 
If $T \subeq hA$, then by Lemma \ref{lem:index} we have $(h-1)( (F-b) \cup\{0\})+H=G$ and therefore $h \ge h_1 +1$. 
We will now prove that $T \subeq (h_1 + h_2) A$. We have
\begin{eqnarray*}
(h_1 + h_2) (A-b) & = & \bigcup_{i=0}^{h_1 + h_2} \left( i(F-b)+(h_1 + h_2-i)(B-b) \right) \\
& \supset & \bigcup_{i=0}^{h_1}\left( i(F-b)+h_2(B-b) \right) \qquad \textup{since $0 \in B-b$} \\
&=& h_2(B-b)+h_1((F-b)\cup\{0\}).
\end{eqnarray*}
Since $h_2 (B-b)$ misses only finitely many elements of $T\cap H$ and $h_1 ((F-b)\cup\{0\})$ meets every coset of $H$, by Lemma \ref{lem:translatable} part (6) we know that $T \subeq (h_1 + h_2) (A-b)$, and $T \subeq (h_1 + h_2) A$ by translatability.
\end{proof}
We are now ready to prove Theorem \ref{th:EGT}.
\begin{proof}[Proof of Theorem \ref{th:EGT}]
Fix $h\ge 2,k\ge 1$.
In order to show that $E_T(h,k) = E_G(h,k)$, we will show that $E_T(h,k)\le E_G(h,k)$ and 
$E_G(h,k)\le E_T(h,k)$.
 
Let us first prove that $E_T(h,k)\le E_G(h,k)$.
Let $A$ be a basis of $T$ of order at most $h$. 
Our aim is to prove that $A$ has at most $E_G(h,k)$ essential subsets of cardinality $k$.

By Theorem \ref{th:e1}, we already know that $A$ has finitely many essential subsets. 
Let $F$ be the union of all essential subsets of $A$. 
From now on, and since the desired inequality readily holds true otherwise, we assume that $F$ is nonempty.
Let $B= A \setminus F$ and $H = \langle B - B \rangle$. 
By definition, $A= F\cup B$ and taking an arbitrary element $b \in B$, we have $B \subset H+b$.

By Lemma \ref{lem:index}, $H$ is a subgroup of finite index of $G$ so that Lemma \ref{lem:twobases} applies to the partition $A=F \cup B$. 
It follows that, since $A$ is a $G$-basis of $T$, the condition (2a) of Lemma \ref{lem:twobases} is satisfied, that is to say $B-b$ is an $H$-basis of $T \cap H$.

Also, let us prove that $F \cap (H+b) = \emptyset$. 
Assume on the contrary that there is an element $x \in F \cap (H+b)$. 
Then, there exists an essential subset $E'$ of $A$ such that $x \in E'$. 
Since $E' \subset F$, we obtain $b \in A \setminus F \subset A \setminus E'$. 
Letting $H_{E'}=\langle A \setminus E' - A \setminus E' \rangle$, we have $H \subset H_{E'}$, that is $H+ b \subset H_{E'}+b$.
By Corollary \ref{cor:ess}, $G$ is generated by $H_{E'} \cup \{x-b\}$.
Yet $x-b \in H_{E'}$ which yields $G=H_{E'}$, a contradiction.
 
By Lemma \ref{lem:index}, $(h-1)(F\cup\{b\})+H=G$.
Let $A'= F\cup (H+b) \subset G$. 
Then $A'$ is a basis of $G$ of order at most $h$. 
Also, $\langle A' \setminus F - A' \setminus F \rangle=H$ is a subgroup of finite index of $G$ so that Lemma \ref{lem:twobases} applies to the partition $A'=F \cup (H+b)$.
Finally, the condition (2a) of Lemma \ref{lem:twobases} is trivially satisfied in this case.

Now let $E \subset F$ be any subset.
We know that $B \subset H+b$ and $E\cap (H+b) =\emptyset$. 
Since $H$ is a subgroup of finite index of $G$, it follows that $A \setminus E= (F\setminus E) \cup B$ and $A'\setminus E= (F\setminus E) \cup (H+b)$ are two partitions to which Lemma \ref{lem:twobases} applies. 
Note also that the condition (2a) of that lemma has already been proved to hold in both cases.
This gives
\begin{equation*}
 A \setminus E \textup{ is a $G$-basis of } T \iff \langle F\setminus E - b +H \rangle = G
 \iff A' \setminus E \textup{ is a basis of } G.
\end{equation*}
Consequently, each essential subset of $A$ (all of which are subsets of $F$) is an essential subset of $A'$.
Now $A'$ has at most $E_G(h,k)$ essential subsets of cardinality $k$ by definition, whence $E_T(h,k)\le E_G(h,k)$.

To prove that $E_G\le E_T$, we argue similarly; thus let $A$ be a basis of $G$ of order at most $h$ and let $F$ be the union of its essential subsets. 
From now on, and since the desired inequality readily holds true otherwise, we assume that $F$ is nonempty.
Using Lemma \ref{lem:translatable} part (3), by translating $A$ by some $t \in T$, and since translations preserve bases and the number of essential subsets, we may assume that $F \subset T$. 
By Lemma \ref{lem:eg}, the subgroup $H=\langle A\setminus F-A\setminus F\rangle$ of $G$ is proper and of finite index, and $A=F\cup B$ where $B =  A\setminus F \subset x+H$ for some $x\in G$. 
We may assume $x\in T$ by Lemma \ref{lem:translatable} part (4). 
We have again $(h-1)(F\cup\{x\})+H=G$ by Lemma \ref{lem:index}.
Let $A'=F\cup (x+T\cap H)\subset T$. Then $hA' \supset (h-1)(F\cup\{x\}) + T \cap H \sim T$ by Lemma \ref{lem:translatable} part (6). 
Using Lemma \ref{lem:twobases} in the same way as before, we see that if $E \subset F$ then
\begin{equation*}
 A \setminus E \textup{ is a basis of } G \iff \langle F\setminus E - x +H \rangle = G
 \iff A' \setminus E \textup{ is a basis of } T.
\end{equation*}
This shows that all essential subsets of $A$ are essential subsets of $A'$, so $A$ has at most $E_T(h,k)$ essential subsets of size $k$, and finally $E_G(h,k)\le E_T(h,k)$. This concludes the proof.
\end{proof}
\subsection{Discussion of finite quotients and lower bounds}
\label{sec:lowbounds}
In this section we exhibit a connection between the function $E_G(h,k)$ and finite quotients of $G$, which we use to prove parts $(ii)$ of Theorem \ref{th:ElowBounds}
and $(iii)$ of Theorem \ref{th:e2}.

Let $G$ be a (possibly finite) abelian group.
We say that $A\subset G$ is a \textit{nice basis} of $G$ if $hA=G$  and a \textit{nice weak basis} if $\bigcup_{i=0}^hiA= G$ for
some $h\in\N$. Also if $G$ is finite, note that $A\subset G$ is a nice basis if and only if $\langle A-A\rangle=G$, and a nice weak basis if and only if $\langle A\rangle =G$.

A finite subset $F$ of a nice basis $A$ is said to be \textit{nicely exceptional} if $A\setminus F$ is no longer a nice basis, and \textit{nicely essential} if it is minimal for this property.
Similarly, a  finite subset $F$ of a nice weak basis $A$ is called nice-weakly exceptional if
$A\setminus F$ is no longer a nice weak basis, and \textit{nice-weakly essential} if it is additionally minimal for this property.

We define $E_G^*(h,k)$ (resp. $E_G^*(h,\le k)$) as the maximal number of nice-weakly essential sets of cardinality $k$ (resp. at most $k$) a nice weak basis $A$ of order at most $h$ of $G$ may have. 
\begin{proposition}
\label{prop:Equotients}
Let $G$ be an infinite abelian group and $h\ge 2, k\ge 1$ integers.
Then $E_G(h,k)\ge \displaystyle \max_{[G:H]<\infty}E_{G/H}^*(h-1,k)$ and
$E_G(h,\le k)= \displaystyle \max_{[G:H]<\infty}E_{G/H}^*(h-1,\le k)$.

\end{proposition}
Note that by Theorem \ref{th:ElowBounds}$(i)$, we already know that if $G$ does not have subgroups of finite index, then $E_G(h, k) = E_G(h, \le k) =0$.
\begin{proof}
Let $H$ be a finite index subgroup of $G$.
Let $A\subset G/H$ be a nice-weak basis of order at most $h-1$ which has $E_{G/H}^*(h-1,k)$
essential subsets of cardinality $k$. We may suppose that $0\notin A$.
Let $\tilde{A}\subset G$ be a set of representatives of $A$; in particular $\tilde{A}\cap H=\emptyset$.
Let $B=\tilde{A}\cup H$. It is a basis of order  at most $h$ of $G$.
For any subset $F\subset A$ of cardinality $k$, let 
$\tilde{F}\subset \tilde{A}$ be the set of representatives of the elements of $F$ inside $\tilde{A}$.
Applying Lemma \ref{lem:twobases} (with the roles of $B, F, A$ in that lemma played by $H, \tilde{A} \setminus \tilde{F}, (\tilde{A} \setminus \tilde{F}) \cup H$ respectively),
we see that $F$ is nice-weakly exceptional in $A$ $\Leftrightarrow$ $\langle A \setminus F \rangle \neq G/H$ $\Leftrightarrow$ $\tilde{F}$ is exceptional in $B$.
In particular $F$ is nice-weakly essential in $A$ if and only if $\tilde{F}$ is essential in $B$.
Thus $E_G(h,k)\ge E_{G/H}^*(h-1,k)$. The proof of $E_G(h,\le k)\ge E_{G/H}^*(h-1,\le k)$ runs along the same lines.

We will now prove that $E_G(h,\le k)\le E_{G/H}^*(h-1,\le k)$ for some subgroup $H$ of finite index of $G$. Suppose $A\subset G$ is a basis of order at most $h$ which has $E_G(h,\le k)$ essential subsets of cardinality at most $k$. Let $F$ be the union of these essential subsets and 
$H=\langle A\setminus F-A\setminus F\rangle$, a subgroup of finite index by Lemma \ref{lem:index}.
Thus $A\subset F\cup (x+H)$ for some $x\in G$.
Upon translating, we may assume that $x=0$. For any subset $X \subset A$ let $\overline{X} \subset G/H$ denote the projection of $X$ on $H$.
Then the projection $\overline{A} = \overline{F}\subset G/H$ is a nice weak basis of order at most $h-1$ of $G/H$. 
Let $E$ be any essential subset of $A$.

\medskip
\noindent \textbf{Claim 1.} $\overline{A\setminus E} \cap \overline{E} = \emptyset$. Equivalently, $(A \setminus E) \cap (E+H) = \emptyset$.

\begin{proof}[Proof of Claim 1.]
We have $A\subset E\cup H_E$, where $H_E=\langle A\setminus E-A\setminus E\rangle $ is a proper 
subgroup of $G$ containing $H$. Thus $E\supset A\setminus H_E$ and by minimality $E=A\setminus H_E$. Consequently, $E + H_E = (A\setminus H_E) + H_E \subset G\setminus H_E$. Thus
\[
(A \setminus E) \cap (E+H) \subset (A \setminus E) \cap (E+H_E) \subset H_E \cap (G \setminus H_E) = \emptyset
\]
and Claim 1 is proved.
\end{proof}

\noindent \textbf{Claim 2.} $\overline{E}$ is a nice-weakly essential subset of $\overline{A}$.

\begin{proof}[Proof of Claim 2.]
Observe that by Claim 1, $\overline{A}\setminus\overline{E} = \overline{A\setminus E}\subset H_E/H$
so $\langle \overline{A}\setminus\overline{E} \rangle \neq G/H$ and $\overline{E}$ is a nice-weakly exceptional subset of $\overline{A}$. Let us now show that for any $x \in E$, the subset
$\overline{E} \setminus \{ \overline{x} \}$ of $\overline{A}$ is not nice-weakly exceptional, i.e. $\langle (\overline{A} \setminus \overline{E})\cup \{\overline{x}\}\rangle = G/H$. But this follows from the facts that
$\langle \overline{A\setminus E}\rangle=H_E/H$, and the projection of $x$ on $H_E$ generates $G/H_E$ by Corollary \ref{cor:ess}. Thus Claim 2 is proved.
\end{proof}

Further, if $E_1\neq E_2$ are two distinct essential subsets of $A$, then $\overline{E_1}$ and $\overline{E_2}$
are distinct since $A\cap (E_1+H)=E_1$ and $A\cap (E_2+H)=E_2$ by Claim 1. Since the cardinality of $\overline{E}$ does not exceed that of $E$, Claim 2 implies that $E_G(h,\le k)\le E_{G/H}^*(h-1,\le k)$, and we are done.
\end{proof}

\begin{remark}
We note two points.
\begin{itemize}
\item Note that the cardinality of an essential subset may decrease upon projection. 
This is why the equality statement in Proposition \ref{prop:Equotients} applies to the function $E_G(h,\le k)$
 rather than $E_G(h,k)$.
\item This proposition reveals that for each $h$ and $k$, the
invariant $E_G(h,\le k)$ of $G$ is entirely determined by the set of finite quotients of $G$; thus for instance
$E_{\F_2[t]}(h,\le k)=E_{\F_2^\N}(h,\le k)$.
\end{itemize}
\end{remark}
We now exhibit a simple way of bounding $E_{G/H}^*(h-1,k)$ from below.
Let $M(G,k)$ (resp. $M(G,\le k)$) be the number of maximal subgroups of cocardinality $k$ (resp. at most $k$) of a finite group $G$.
\begin{proposition}
\label{prop:decomposition}
Let $G$ be a finite abelian group which admits a decomposition $G=\bigoplus_{i=1}^hG_i$ as a direct sum of $h\ge 1$ subgroups.
Then  $E_G^*(h,k)\ge \sum_{i=1}^hM(G_i,k)$ and $E_G^*(h,\le k)\ge \sum_{i=1}^hM(G_i,\le k)$.
Both inequalities are equalities when $h=1$.
\end{proposition}
It would be interesting to know whether equality always holds for some decomposition 
$G=\bigoplus_{i=1}^hG_i$ of $G$.
\begin{proof}
The hypothesis implies that $A=\bigcup_{i=1}^h G_i$ is a basis of order $h$.
We claim that its nice-weakly essential subsets are precisely the sets of the form $G_i\setminus K$ where $i\in[h]$ and
$K$ is a maximal proper subgroup of $G_i$.
First, let $i\in[h]$ and
$K$  a  proper subgroup of $G_i$. Let $F=G_i\setminus K$. Then $\langle A\setminus F\rangle=\bigoplus_{j\neq i}G_j \oplus K\neq G$ so $F$ is nice-weakly exceptional. If $F$ is additionally maximal, then $F$ is
essential.
Conversely, let $F\subset A$ be nice-weakly exceptional, 
and let $F_i=F\cap G_i$ so $F=\bigcup_{i=1}^h F_i$. Then $A\setminus F=\bigcup_{i=1}^h G_i\setminus F_i$
and $\langle A\setminus F\rangle=\bigoplus_{i=1}^h \langle G_i\setminus F_i\rangle$.
Therefore at least one $i\in [h]$ must satisfy $\langle G_i\setminus F_i\rangle\neq G_i$.
Now  suppose $F$ is nice-weakly essential, so by minimality exactly one $i\in [h]$ must satisfy
$\langle G_i\setminus F_i\rangle\neq G_i$, and finally
 $F=G_i\setminus K$ for some $i$ and some maximal subgroup $K$ of $G_i$. 
Therefore $A$ has  $\sum_{i=1}^hM(G_i, k)$ essential subsets of cardinality $k$ and $\sum_{i=1}^hM(G_i,\le k)$ essential subsets of cardinality at most $k$.
 This proves both inequalities.
 
 Since a nice weak basis of order 1 of $G$ is precisely $G$ or $G\setminus \{0\}$, we obtain the equality when $h=1$.
\end{proof}
 
 The two propositions above imply that $E_G(2,\le k)=\displaystyle \max_{[G:H]<\infty}M(G/H,\le k)$.
 So there remains to understand $M(G,\le k)$ for finite abelian groups.
 \begin{proposition}
 \label{prop:numberMaxSubgroups}
 For any finite abelian group and integer $k\ge 1$, we have
 $M(G,\le k)\le 2k-1$ and equality holds if, and only if, $G=\F_2^d$ and $k=2^{d-1}$ for some $d\ge 1$, in which case we even have $M(G,k)=2k-1$.
 \end{proposition} 
\begin{proof}
The basic theory of finite abelian groups indicates that 
a subgroup is maximal if, and only if, it has cardinality $\abs{G}/p$ for some prime $p$ dividing $ \abs{G}$,
thus $\abs{G}/p\ge \abs{G}-k$. In particular $M(G,k)=0$ unless $\abs{G}\le 2k$, which we henceforth
suppose. Further, the number of subgroups of index $p$ equals the number of subgroups of cardinality $p$. For each prime $p\le \frac{\abs{G}}{\abs{G}-k}$, let $G_p=\{x \in G : px=0\}$; the number of subgroups of order $p$ of $G$ is $\frac{\abs{G_p}-1}{p-1}$.
Therefore,
$$
M(G,\le k)=\sum_{p\le \frac{\abs{G}}{\abs{G}-k}}\frac{\abs{G_p}-1}{p-1}\le \sum_{p\le \frac{\abs{G}}{\abs{G}-k}} (\abs{G_p}-1).
$$
Now we invoke the simple inequality $\sum_{i=1}^n (a_i-1)\le \prod_{i=1}^na_i-1$, valid for any $n$-tuple of real numbers satisfying $a_i\ge 1$, where equality holds if and only if $a_i=1$ for all but at most one $i\in [n]$.
Finally we apply the fact that $\prod_p\abs{G_p}\le\abs{G}$,
which follows from the fact that the subgroups $G_p$ of $G$ are in direct sum,
the already derived condition 
$\abs{G}\le 2k$, and conclude that $M(G,\le k)\le 2k-1$. The equality case follows readily from the conjunction of all four inequalities used in the proof.
\end{proof}
The proof also reveals that if $G$ has no nontrivial element of order less than $p$, we have
$M(G,\le k)\le kp/(p-1)-1$ where equality holds if, and only if, $\abs{G}=\F_p^d$ and $k=p^d-p^{d-1}$.

We now prove Theorem \ref{th:ElowBounds} $(ii)$ and Theorem \ref{th:e2} $(iii)$.

\begin{proof}[Proof of Theorem \ref{th:ElowBounds} (ii) and Theorem \ref{th:e2} (iii)]
For Theorem \ref{th:ElowBounds} $(ii)$, we consider $k=2^{r-1}$ for some $r \ge 1$. By hypothesis, there is a subgroup $H$ of $G$ such that $G/H \cong \F_2^{(h-1)r}$. By Proposition \ref{prop:Equotients}, 
$E_G(h,k) \ge E^*_{G/H}(h-1,k)$. We write $\F_2^{(h-1)r}=\bigoplus_{i=1}^{h-1} G_i$ where 
$G_i\cong \F_2^r$. By Propositions \ref{prop:decomposition} and \ref{prop:numberMaxSubgroups}, we have $E^*_{G/H}(h-1,k) \ge (h-1) M(\F_2^r, 2^{r-1}) = (h-1)(2k-1)$.

For Theorem \ref{th:e2} $(iii)$, Propositions \ref{prop:Equotients} and \ref{prop:decomposition} imply that
\[
E(2, \le k) = \displaystyle \max_{[G:H]<\infty} E^*_{G/H} (1, \le k) = \max_{[G:H]<\infty} M(G/H, \le k).
\]
The last quantity is $\le 2k-1$ by Proposition \ref{prop:numberMaxSubgroups}.
\end{proof}

\section{The function $X_T(h,k)$} \label{sec:x} 
\subsection{Upper bounds}
\label{sec:Xupper}
We fix a translatable semigroup $T$ of Grothendieck group $G=G_T$ and an invariant mean $\Lambda$ on $T$. 
By Lemma \ref{lem:extension}, we extend it to an invariant mean on $G$ by letting $\Lambda(f)=\Lambda(f |_T)$ for any $f\in \ell^\infty(G)$, where $f |_T$ is the restriction of $f$ to $T$.
For a set $A \subset G$, we refer to $d(A)=\Lambda(1_A)$ as the ``density'' of $A$. 
Note that $d(T)=1$.
We first prove some lemmas on the densities of certain sumsets.
\begin{lemma} \label{lm:densityIncrement}
Let $B,C \subset G$. 
Then either $d(B+C)\ge 2d(C)$ or $B-B\subset C-C$.
\end{lemma}

\begin{proof}
Suppose there are two distinct elements $b,b'$ of $B$ such that $b+C$ and $b'+C$ are disjoint.
Then $d(B+C)\ge d((b+C)\cup (b'+C))=2d(C)$.
Otherwise, for any $b\neq b'$ of $B$ we have $(b+C)\cap (b'+C)\neq \emptyset$, which implies that $b-b'\in C-C$, so that $B-B\subset C-C$.
\end{proof}
We shall deduce by iteration the following corollary.
\begin{corollary}
\label{cor:iter}
Let $A \subset G$. Let $r\ge 1$ be an integer.
For any $i\ge 0$, let $s_i=2^{i} r+ 2^{i} -1$.
Then either $d(s_iA)\ge 2^id(rA)$ or $i\ge 1$ and $s_{i-1}(A-A)=\langle A-A\rangle$.
\end{corollary}
\begin{proof}
We argue by induction. For $i=0$ the claim is trivial.

Fix some $i\ge 0$ and let us show that either $d(s_{i+1}A)\ge 2^{i+1}d(rA)$ or $s_{i}(A-A)=\langle A-A\rangle$.
We apply Lemma \ref{lm:densityIncrement}, to $C=s_iA$ and $B=(s_i+1)A$.
Then $B+C=s_{i+1}A$. If $B-B\subset C-C$, we have for any $s\ge s_i$ the inclusion $s(A-A)\subset s_i(A-A)$.
Since $\langle A- A \rangle = \bigcup_{j=1}^\infty j(A-A)$, this implies that $s_i(A-A)=\langle A-A\rangle$.
 
Otherwise, we must have $d(s_{i+1}A)=d(B+C)\ge 2d(s_iA)$.
Further, note that $s_i(A-A)\neq \langle A-A\rangle$, and therefore for any $s\le s_i$ we know that $s(A-A)\neq \langle A-A\rangle$.
If $i=0$ we are done. Otherwise, applying the induction hypothesis, we see that $d(s_j)\ge 2d(s_{j-1}A)$ for any $j\le i$.
By a straightforward induction, we conclude that $d(s_{i+1}A)\ge 2^id(rA)$.
\end{proof}

We now show that if $d(hA)>0$, then $A-A$ must be a basis of bounded order of the group it generates.
\begin{lemma}
\label{lm:densityIncrement2}
Suppose $A \subset G$, $h \ge 1$ and $d(hA) = \alpha>0$. 
Then there exists $s \le \frac{1}{\alpha}(h+1)-1$ such that $sA-sA=\langle A-A\rangle $.
\end{lemma}
\begin{proof}
We apply Corollary \ref{cor:iter} to the set $hA$, the integer $r=h$ and $i=i_0$ the smallest integer such that $2^{i_0}\alpha >1$.
Since the density cannot exceed 1, we have $s_i(A-A) = \langle A-A \rangle$ where $s_i = 2^{i_0-1} h+ 2^{i_0-1} -1 \le \frac{1}{\alpha}(h+1)-1$ (since $\frac{1}{\alpha} \ge 2^{i_0-1}$). 
This yields the desired conclusion.
\end{proof}

The following lemma can be regarded as an analogue of \cite[Lemma 3]{nn}.
\begin{lemma}
\label{lm:NathNashLike}
Let $B\subset G$ satisfy $\langle B-B\rangle=G$.
Suppose there exist $h, m \ge 1$ and $x_1,\ldots,x_m$ in $T$ such that $T \subeq \bigcup_{i=1}^m (x_i+hB)$.
Then $B$ is a $G$-basis of $T$ of order at most $ h+ m^2 (h+1) -m$.
\end{lemma}
The term $m^2$ (whereas one could hope for $m$ instead) is what is ultimately causing our bounds for $X_G(h,k)$ to be large; we do not know whether it is optimal.
\begin{proof}
The hypothesis and axioms of a density imply $d(hB)\ge 1/m$.
By Lemma \ref{lm:densityIncrement2}, we infer that there exists $s \le m(h+1)-1$ such that $sB-sB=\langle B-B\rangle=G$.
Thus, for each $1 \le i \le m$ we may write $x_i=a_i-b_i$ where $a_i\in sB$ and $b_i\in sB$. 
Hence
\[
T\subeq \bigcup_{i=1}^m (hB+a_i-b_i).
\]
By adding $\sum_{i=1}^m b_i$ to both sides and using translatability, we have
\[
T \subeq \bigcup_{i=1}^m  (hB+a_i+\sum_{j\neq i}b_j)
\]
which shows that all except finitely many elements of $T$ can be expressed as a sum of $h + ms$ elements of $B$. Since $h+ms \le h+ m^2 (h+1) -m$, we are done.
\end{proof}

We may now deal with the effect of removing a regular subset from a basis.

\begin{proof}[Proof of Theorem \ref{th:x2}]
Let $A$ be a $G$-basis of order at most $h$ and $F \subset A$ be a regular subset of cardinality $k$. 
Let $B=A \setminus F$. 
Since $F$ is regular, by Lemma \ref{lem:eg}, we have $\langle B-B \rangle =G$. 
We observe that
\begin{align}
\label{eq:decompo}
T & \subeq hB \cup ((h-1)B+F)\cup \cdots \cup (B+(h-1)F). 
\end{align}
Let $b\in B$.
Since $iB\subset hB-(h-i)b$, we have $iB+hb\subset hB+ib$ and by translatability
$$T \subeq (hB+hb) \cup (hB+F+(h-1)b)\cup \cdots \cup (hB+(h-1)F+b).$$

Therefore we may apply Lemma \ref{lm:NathNashLike} with
\begin{equation}
\label{eq:boundCard}
m = \abs{\displaystyle\bigcup_{j=1}^{h-1} (jF+(h-j)b)} \le \abs{\left\{ (t_1, \ldots, t_k) \in \N^k : \sum_{i=1}^k t_i \le h-1 \right\}} =\binom{h+k-1}{k}.
\end{equation}
We infer that $B$ is a $G$-basis of order at most $$(h+1) \binom{h+k-1}{k} ^2 - \binom{h+k-1}{k} +h = \frac{h^{2k+1}}{k !^2} (1 + o_k(1)),$$ 
which is the desired result.
\end{proof}

\begin{remark}
 In the case $k=1$, the proof gives the bound $X_T(h) \le (h+1)h^2$. By bounding $d(iB)$ in terms of $d(hB)$ using Corollary \ref{cor:iter} for each $i \le h$, one can prove a better bound $X_T(h) \le \left( \frac{2}{3} + o(1) \right) h^3$. 
\end{remark}

\begin{remark}
For fixed $k$ and $h \rightarrow \infty$, we do not know if the estimate $\ord^*_T(B) = O_k(h^{2k+1})$ is best possible. On the other hand, we see that $\ord^*_G(B-B) 
\le (h+1) \binom{h+k-1}{k}-1 = O_k(h^{k+1})$ (by Lemma \ref{lm:densityIncrement2} and the fact $d(B)\ge 1/\binom{h+k-1}{k}$). This estimate is best possible in terms of $h$, as shown by the following example. Let $A = \{0, 1, b, \ldots, b^k\} \cup b^{k+1} \N$. Then $A$ is a basis of order $h=(b-1)(k+1)$ of $\N$. Let $F=\{ b, \ldots, b^k \}$, then $B = A \setminus F = \{ 0, 1\} \cup b^{k+1} \N$ 
and $B-B = \{ 0, \pm 1\} \cup b^{k+1} \Z$. Then $B-B$ is a basis of $\Z$ of order $\le \frac{b^{k+1}}{2} = O_k(h^{k+1})$. We will use this idea again to prove Theorem \ref{th:XlowBounds}.
\end{remark}

In the case of $\sigma$-finite groups, we can do better.
\begin{proof}[Proof of Theorem \ref{th:sigmaFinite}]
Let $T$ be a $\sigma$-finite infinite abelian group.
Let $(G_n)_{n\ge 0}$ be a nondecreasing sequence of subgroups such that $T=\bigcup_{n\ge 0} G_n$.
For $C\subset T$, let $\undd{C}=\limsup_{n\to\infty}\frac{|C\cap G_n|}{|G_n|}$ be its upper asymptotic density.
Let $A$ be a basis of $G$ of order at most $h\ge 2$.
Let $F$ be a regular subset of $A$ of cardinality $k$ and $B=A\setminus F$.
Upon translating we may assume that $0\in B$. Note that $\langle B\rangle = \langle B-B\rangle =T$ by Lemma \ref{lem:eg}.
By equation \eqref{eq:decompo}, we have $\undd{hB+\bigcup_{j=0}^{h-1}jF}=\undd{T}=1$.
Note that for any two subsets $X,Y$ of $T$, for any $\epsilon >0$, we have
$$
\frac{|(X\cup Y)\cap G_n|}{|G_n|}\le \frac{|X\cap G_n|+|Y\cap G_n|}{|G_n|}\le \undd{X}+\undd{Y}+\epsilon.
$$
Taking the upper limit, we find that
$\undd{X\cup Y}\le \undd{X}+\undd{Y}+\epsilon$.
Finally, letting $\epsilon$ tend to 0, we see that $\undd{X\cup Y}\le \undd{X}+\undd{Y}$.

Because of the translation-invariance of the density, the just obtained inequality and equation \eqref{eq:boundCard}, we infer that $\undd{hB}\binom{h+k-1}{k}\ge 1$.
We are now in position to apply \cite[Theorem 1]{hr}, which yields that $hB$ is a basis of $\langle hB \rangle=G$ of order at most $1+2/\undd{hB}\le 1+2\binom{h+k-1}{k}=\frac{h^k}{k!}+O(h^{k-1})$.
Therefore, $B$ itself is a basis of order at most $h\,\ord_G^*(hB)\le 2\frac{h^{k+1}}{k!}+O(h^k)$.
\end{proof}

\begin{remark}
Instead of appealing to \cite[Theorem 1]{hr}, we could have used Kneser's theorem for the lower asymptotic density \cite{BH} and the fact that any set $A$ of lower asymptotic density larger than 1/2 satisfies $A+A\sim G$ and argued like Nash and Nathanson in the integers.

Note that a Kneser-type theorem is available in any countable abelian group $G$ for the upper \textit{Banach} density \cite{Griesmer}.
However, that density has the drawback that a set $A\subset G$ satisfying $\band{A}>1/2$, even $\band{A}=1$, may not be a basis of any order of the group it generates.
For instance, take $B=\bigcup_{i\ge 1} [2^i,2^i+i)\subset \Z$ and $A=B\cup (-B)$; it generates $\Z$ but is far too sparse to be a basis of $\Z$, of any order. 
Yet its upper Banach density is 1.
\end{remark}
We conclude the section with the case of infinite abelian groups of finite exponent.

\begin{proof}[Proof of Theorem \ref{th:x3}]
Let $G$ be an infinite abelian group of exponent $\ell$.
For part (1), we proceed identically to the proof of Theorem \ref{th:x2} with the group $G$ in the place of $T$. 
The difference is that, since $G$ has exponent $\ell$,
$$m = \abs{\displaystyle\bigcup_{j=1}^{h-1} jF} \le \abs{\left\{ (t_1, \ldots, t_k) \in \N^k : t_i \le \ell-1, \sum_{i=1}^k t_i \le h-1 \right\}} \le \ell^k.$$
Thus by Lemma \ref{lm:NathNashLike}, $B$ is a basis of order at most $ (h+1) \ell^{2k} - \ell^k +h$
as desired.

As for part (2), we will generalize the argument in \cite[Theorem 5]{llp}. 
Suppose $F=\{a\}$. 
By translating $A$ by $-a$ if necessary, we may assume that $a=0$. 
Since $G$ has exponent $\ell$, we have $sB \subset (s+\ell)B$ for any $s$. Therefore,
\begin{equation} \label{eq:lcup}
G \sim \bigcup_{i=1}^h iB \sim \bigcup_{i=h-\ell+1}^h iB.
\end{equation}
For any $x \in G$, since $B$ is infinite, $(x - B) \cap \bigcup_{i=h-\ell+1}^h iB$ is nonempty and therefore
$$G = \bigcup_{i=h - \ell +2}^{h+1} iB.$$

We now claim that there are $u, v$ such that $h+2 \le u < u+v \le h+\ell +1$,  $uB \cap (u+v)B \neq \emptyset$ and $\gcd(v, \ell)=1$. 
Suppose for a contradiction that this is not true. 
Then we have disjoint unions
$$G = \bigcup_{i \in I_1 +\ell} iB \sqcup  \bigcup_{i \in I_2 + \ell} iB$$
and
$$G = \bigcup_{i \in I_1} iB \sqcup  \bigcup_{i \in I_2} iB,$$
where 
$$I_1:= \{ j \in [h-\ell +2, h+1] : p \mid j\}$$
and 
$$I_2:= \{ j \in [h-\ell +2, h+1] : p \nmid j \},$$
where $p$ is the unique prime divisor of $\ell$. 
It follows that $\bigcup_{i \in I_1} iB = \bigcup_{i \in I_1 +\ell} iB$. 
By repeatedly adding $\ell B$ to both sides, we have $\bigcup_{i \in I_1} iB = \bigcup_{i \in I_1 + s \ell} iB$ for any $s \ge 1$. 
For $s$ sufficiently large, this implies that $\bigcup_{i \in I_1} iB = G$ (since we already know that $B$ is a basis). 
This is a contradiction and the claim is proved.

We now proceed similarly to the proof of Lemma \ref{lem:weakbasis}. If $c \in uB \cap (u+v)B$, then
$$(\ell-1) c \in (\ell-1)u B \cap ((\ell -1) u+v) B \cap \cdots \cap (\ell-1)(u+v)B.$$

Let $y_i=(\ell-1)u+iv$.
For each $i\in [0,\ell -1]$, there exists $x_i\in [(\ell-1) (u+v-1), (\ell-1)(u+v)] $ satisfying $x_i\equiv y_i \mod \ell$ and $x_i>y_i$.
Further, since $\gcd(v,\ell)=1$, we have $\{x_0,\ldots,x_{\ell -1}\}=[(\ell-1) (u+v-1), (\ell-1)(u+v)]$.
Therefore,
\begin{equation} \label{eq:lcap}
(\ell-1) c \in \bigcap_{i=(\ell-1) (u+v-1)}^{(\ell-1) (u+v)} iB. 
\end{equation}
For all but finitely many $x \in G$, from \eqref{eq:lcup} and \eqref{eq:lcap}, we have
$$x = (x- (\ell-1) c ) + (\ell-1) c \in \left( (\ell-1) (u+v) + h - \ell +1 \right) B.$$
Therefore, $B$ is a basis of order at most $ (\ell-1) (u+v) + h - \ell +1 \le (\ell-1) (h + \ell +1) + h - \ell +1 = h\ell + \ell^2 - \ell$.
\end{proof}

\begin{remark} 
What we need about $\ell$ in the proof is that whenever $\gcd(a, \ell)=1$ and $\gcd(b, \ell) \neq 1$, then $\gcd(a-b, \ell) =1$. 
Obviously, prime powers are the only integers having this property.
\end{remark} 

\subsection{Lower bounds}
\label{sec:XlowBounds}
We will again modify constructions from quotients, but in a different way than in Section \ref{sec:lowbounds}.
Let $X_G^*(h,k)$ be the maximal order of a nice basis which is included and has cocardinality $k$ in a nice basis of order at most $h$. Nice bases are defined at the beginning of Section \ref{sec:lowbounds}.
\begin{proposition}
\label{prop:Xquotients}
Let $G$ be an infinite abelian group and $H$ an  infinite subgroup of $G$. Then
$X_G(h,k)\ge X^*_{G/H}(h,k)$.
\end{proposition}
\begin{proof}
Let $A$ be a nice basis of order (at most) $h$ of $G/H$ which contains a subset $F$ of cardinality $k$ such that $A\setminus F$ is a nice basis of order $X^*_{G/H}(h,k)$ of $G/H$.

Let $\pi$ be the projection $G\rightarrow G/H$ and let $\tilde{F}$ be a set of representatives of
$F$ in $G$.
Then let $B=\tilde{F}\cup \pi^{-1}(A\setminus F)$.
Note that $A\setminus F$ is not empty since it is a basis.
Then $B$ is a basis of order (at most) $h$ of $G$
and for any $h'\ge 1$, we have $h'\pi^{-1}(A\setminus F)\sim G$
if and only if $h'(A\setminus F)= G/H$, which by hypothesis is equivalent to $h'\ge X^*_{G/H}(h,k)$.
So $B\setminus \tilde{F}$ is a basis of order $X^*_{G/H}(h,k)$ of $G$, which concludes.
\end{proof}
Now we prove Theorem \ref{th:XlowBounds}.
\begin{proof}[Proof of Theorem \ref{th:XlowBounds}]
Fix some integer $k$. By hypothesis, there is an infinite set $I$ such that for any $N \in I$, $G$ has a quotient isomorphic to $\Z / N \Z$.  
Let $N \in I$ and $b$ satisfy $(b-1)^{k+1}<N\le b^{k+1}$.
If $N$ is sufficiently large, then so is $b$, and $b^k<(b-1)^{k+1}<N$.
Consider the nice basis
$A=\{0,1,b,\ldots,b^k\}$ of order at most $h:=(b-1)(k+1)$ of $\Z/N\Z$.
Then $F=\{b,\ldots,b^k\}$ has the property that $A\setminus F = \{0,1\}$ is a nice basis of order 
$N-1\ge (b-1)^{k+1}\gg_k h^{k+1}$. Proposition \ref{prop:Xquotients} then implies that $X_G(h,k) \geq X_{\Z / N \Z}(h,k) \gg_k h^{k+1}$, proving the first part of the theorem.

In the regime where $h$ is fixed and $k$ tends to infinity, we use the following fact: there exists a constant $c_h>0$ such that for any large $x$,
there exists a set $A\subset [0,x)$ of integers such that $[0,x)\subset hA$ and $\abs{A}\le c_hx^{1/h}$. Such a set is referred to as a thin basis in the literature and was first constructed by Cassels \cite{cassels}. In particular $0$ and $1$ lie in $A$.
Now consider $A$ as a nice basis of $\Z/x\Z$ of order $h$.
Consider $F=A\setminus \{0,1\}$ of cardinality $k\le c_hx^{1/h}$, so $A\setminus F$ is a nice basis of order
$x\gg_h k^{h}$. One last appeal to Proposition \ref{prop:Xquotients} concludes then the proof of Theorem \ref{th:XlowBounds}.
\end{proof}

\section{The function $S_T(h,k)$}
\label{sec:s} Again, in this section we fix a translatable semigroup $T$ and an invariant mean $\Lambda$ on $T$. 
Recall that $\Lambda$ extends to an invariant mean on $G=G_T$ by setting $\Lambda(f) = \Lambda (f |_T)$ for all $f \in \ell^\infty (G)$, where $f |_T$ is the restriction of $f$ to $T$. 
For $A \subset G$, we write $d(A)=\Lambda(1_A)$. 

We first prove the following observation already used in \cite[Section 6]{llp}.
\begin{lemma}\label{lem:lowdensity} 
Suppose $A \subset G$, $a \in A$ satisfy $T\subeq hA $ and $d(T \setminus h (A \setminus \{a\} ) ) < \frac{1}{h}$. Then 
$T\subeq 2h (A \setminus \{a\} ) $. 
\end{lemma}
\begin{proof}
Let $a_0$ be an element in $A \setminus \{a\}$. Let $B = T \setminus h (A \setminus \{a\})$, then $d(B) <1/h$. Since $d$ is translation-invariant, we have 
\[
\sum_{i=0}^{h-1} d(B+(h-i)a + i a_0) < 1
\] 
and consequently there are infinitely many $x \in T$ such that $x + h(a+a_0) \not \in B+(h-i)a + i a_0 $ for all $i=0, 1, \ldots, h-1$. In other words, $x + ia + (h-i)a_0 \in h(A \setminus \{a\})$ and $x + i(a-a_0) \in h(A \setminus \{a\} -a_0)$ for all $i=0, 1, \ldots, h-1$. 

Now for all but finitely many $t \in T$, we have $t-x \in h(A-a_0)$ and $t-x \neq h(a-a_0)$. 
If $i$ is the number of occurrences of $a-a_0$ in some representation of $t-x$ as a sum of $h$ elements of $A-a_0$, then $0 \le i \le h-1$ and $t-x - i(a-a_0) \in (h-i) ( A \setminus \{a\} - a_0)$. 
Thus 
\[
t = (t-x-i(a-a_0))+(x+i(a-a_0))  \in (2h -i) (A \setminus \{a\} -a_0) \subset 2h (A \setminus \{a\} -a_0),
\] 
and the lemma is proved.
\end{proof}

\begin{proof}[Proof of Theorem \ref{th:s1}] 
We first strengthen slightly an observation already used in \cite[Section 6]{llp}.

\medskip
\noindent \textbf{Claim 1.} For any finite subset $I \subset A$, for all but finitely many $x \in T$, there are at most $h-1$ elements $a \in I$ such that $x \in T \setminus h (A \setminus \{a\})$. 

\begin{proof}[Proof of Claim 1.]
Since $T \setminus hA$ is finite, we may assume $x \in hA$. 
Fix a representation
$$x = a_1 +\cdots + a_h,$$
where $a_i \in A$ for $i=1, \ldots, h$. 
If $x \in T \setminus h (A \setminus \{a\})$, then $a$ must be one of $a_1, \ldots, a_h$. 
This already implies that there are at most $h$ elements $a \in I$ such that $x \in T \setminus h (A \setminus \{a\})$. 
Furthermore, if $x \in T \setminus h (A \setminus \{a\})$ for $h$ elements $a \in I$, then necessarily $x \in hI$. Since $hI$ is finite, this proves the claim.
\end{proof}

Let $I$ be an arbitrary finite subset of $A$.
Let $f(x)=\sum_{a\in I} 1_{T\setminus h(A\setminus \{a\})}(x)$.
Then for all but finitely many $x$, we have $f(x)\le h-1$.
By evaluating $\Lambda(f)$ and the fact that finite sets have density 0, we have the following

\medskip
\noindent \textbf{Claim 2.} For any finite set $I \subset A$, we have $\sum_{a \in I} d\big( T \setminus h (A \setminus \{a\}) \big) \le h-1$. 

\medskip
Suppose now $T$ is a group. We may assume that $\Lambda$ satisfies property (D4) in Section \ref{sec:invariant}. We have

\medskip
\noindent \textbf{Claim 3.} If $B \subset T$ and $d (B) > 1/2$, then $2B = T$.

\medskip
This immediately follows from Lemma \ref{lem:prehistoric}.

Let $J$ be the set of all $a \in A$ such that $\ord^*_T(A \setminus \{a\}) > 2h$. 
For all $a \in J$, we have $d( h (A \setminus \{a\}) ) \le 1/2$ (if not, we will have $2h (A \setminus \{a\})=T$) and therefore $d \big( T \setminus h (A \setminus \{a\}) \big) \ge 1/2$. Since $\sum_{a \in I} d \big( T \setminus h (A \setminus \{a\}) \big) \le h-1$ for any finite subset $I$ of $J$, this shows that $J$ is finite and $|J| \le 2(h-1)$, and the second part of Theorem \ref{th:s1} is proved.

\medskip
For general translatable semigroups, we use Lemma \ref{lem:lowdensity} instead of Claim 3. 
For all $a \in J$, we have $d( h (A \setminus \{a\}) ) \le 1/h$ and therefore $|J| \le h(h-1)$.
\end{proof}

We now generalize these ideas to prove Theorem \ref{th:s2}.

\begin{proof}[Proof of Theorem \ref{th:s2}] 
Let $R$ be the set of all regular pairs $\{ a, b\} \subset A$ such that $\ord^*_T(A \setminus \{a,b\}) > 2 X_T(h)$. 
Also, let $U$ be the set of all regular elements $a \in A$ such that $\ord^*_T(A \setminus \{a \}) > S_T(h)$. 
By Theorem \ref{th:s1} we know that $|U| = O(h^2)$. 

\medskip
\noindent \textbf{Claim 1.} For all but finitely many $x \in T$, there are at most $h X_T(h) (X_T(h)-1)$ pairs $F \in R$ such that $x \in T\setminus h(A\setminus F)$. 
If $T$ is a group then the number of such pairs is at most $2h (X_T(h)-1)$.

\begin{proof}[Proof of Claim 1.]
Since $T \setminus hA$ is finite, we may assume $x \in hA$. Fix a representation
$$x = a_1 +\cdots + a_h,$$
where $a_i \in A$ for $i=1, \ldots, h$. 
If $x \in T\setminus h(A\setminus F)$, then $a_i \in F$ for some $i$. Let $F=\{a_i, b\}$, then $b$ is a regular element of the basis $A \setminus \{a_i\}$ (note that $a_i$ has to be regular in the first place). 
By the definition of $X_T(h)$, we have
$$\ord^*_T(A \setminus \{a_i, b \}) > 2 X_T(h) \ge 2 \ord^*_T(A \setminus \{ a_i \}).$$
By Theorem \ref{th:s1}, there are at most $X_T(h) (X_T(h)-1)$ choices for $b$, and this number can be replaced by $2 (X_T(h)-1)$ if $T$ is a group.
Thus Claim 1 is proved.
\end{proof}

Let $I$ be a finite subset of $R$.
Let $f(x)=\sum_{F\in I}1_{T\setminus h(A\setminus F)}(x).$
Again evaluating $\Lambda(f)$ yields the following bound.

\medskip
\noindent \textbf{Claim 2.} For any finite subset $I \subset R$, we have 
\[\sum_{F \in I} d \big(  T \setminus h(A\setminus F) \big) \le 
\,\begin{cases} 
h X_T(h) (X_T(h)-1)   &\quad  \text{for any $T$},   \\  2h(X_T(h)-1) &\quad  \text{when $T$ is a group}.     
\end{cases}
\]

We are now able to conclude the proof when $T$ is a group.
For all $F \in R$, we have $d \big(  T \setminus h(A\setminus F) \big) \ge 1/2$. 
If not, we will have  $\ord^*_T(h(A\setminus F)) \le 2$ and $\ord^*_T(A\setminus F) \le 2h \le 2 X_T(h)$, which contradicts the definition of $R$. 
This implies that $R$ is finite and furthermore, $|R| \le 4h(X_G(h)-1)$. 

If $T$ is an arbitrary translatable semigroup, then we apply Lemma \ref{lem:lowdensity} to the basis $A \setminus \{a\}$ and get

\medskip
\noindent \textbf{Claim 3.} If $F = \{a,b\}$ is regular, $\ord_T^*(A \setminus \{a\}) = k$, $d(T \setminus k(A \setminus F)) < \frac{1}{k}$, then $\ord^*_T (A \setminus F) \le 2k$. Consequently, if $F \in R$ and $a \not \in U$, then $d( T \setminus h(A \setminus F) ) \ge d(T \setminus 2h(A \setminus F)) \ge \frac{1}{2h}$.

From Claims 2 and 3, the number of pairs $F \in R$, at least one of whose elements is not in $U$, is at most $h X_T(h) (X_T(h)-1) \cdot 2h = O(h^2 X_T(h)^2)$.
Clearly the number of pairs $F \in R$, both of whose elements are in $U$, is $O(h^4)$. 
This concludes the proof of Theorem \ref{th:s2}.
\end{proof}
We point out that the argument used in the proof of Theorem \ref{th:s2} may be applied to bound $S_T(h,k)$ for $k\ge 3$, but it seems to yield bounds which are worse than trivial.

Theorem \ref{th:s1} prompts the following question.

\medskip
\textbf{Question}.
If $T\subeq hA$, then we know that there are at most $h-1$ elements $a \in A$ such that $\ord_T^*(A \setminus \{a\}) = \infty$, and these are characterized by the Erd\H{os}-Graham criterion, i.e. $\langle A \setminus \{a\} - A \setminus \{a\} \rangle \neq G$. 
Can one find a nice  algebraic characterization for elements $a$ for which $\ord_T^*(A \setminus \{a\}) >2h$?

\appendix
\section*{Appendix A. The structure of translatable semigroups} \label{sec:a}
In this appendix we prove the structure result for translatable semigroups.  
\begin{proposition}
\label{prop:structure}
Let $T$ be a translatable semigroup. Then either T is a group (i.e. $T$ equals its Grothendieck group $G_T$), or there exists $x\in T$ and $T\sim C\oplus x\N$, where $C$ is a finite subgroup of $G_T$.
\end{proposition}
\begin{proof}
Suppose that $T$ is not a group.
Let $G=G_T$ be its Grothendieck group. Since $T \neq G$, we have $T \not \subset -T$.
Let $x\in T \setminus (-T)$.  
Then the order of $x$ in $T$ is necessarily infinite, since if $kx=mx$ for some $k>m$ then $-x = (k-m-1)x\in T$, a contradiction.
Therefore $x$ generates an infinite subgroup $x\Z$  of $G$ and also a subsemigroup $x\N^*$ (isomorphic to $\N^*$) of $T$.

Let $R=T\setminus (x+T)$, a finite set.
Let $u\in T$ be arbitrary.
If $u-kx\in T$ for infinitely many positive integers $k$, then since $T \sim u+T$, we have $u-kx \in u+T$ and $-kx\in T$ for some positive integer $k$.
Therefore $-x=-kx+(k-1)x\in T$, which contradicts our hypothesis on $x$.
So we let $u'=u-kx$ where $k$ is the maximum nonnegative integer such that $u-kx\in T$; then $u'\notin x+T$.
As a result, every element of $T$ may be uniquely decomposed as a sum of an element of $R$ and an element of $x\N$, so $T=R+x\N$ and $G=T-T=R-R+x\Z$. Consequently, $x\Z$ has finite index in $G$.

By the classification theorem of finitely generated abelian groups, there exists a finite subgroup $C$ of $G$ such that $G=C\oplus x\Z$.
By Lemma \ref{lem:translatable} part (4), $T \cap (c + x \Z) \neq \emptyset$ for any $c \in C$. On the other hand, we have $T \cap (c + x\Z^{-}) = \emptyset$. 
If not, then since $c$ has finite order, we have $-\ell x \in T$ for some $\ell \in \Z^+$, so $-x = (-\ell x) + (\ell-1)x \in T$, a contradiction. 
Thus for every $c\in C$, there exists a minimal $k\in \N$ such that $c+kx\in T$.
We conclude that $T \sim C\oplus x\N$.
\end{proof}

As a consequence, this structure result implies that any translatable semigroup $T$ admits a basis of any order $h\ge 2$. 
\begin{proposition}
\label{prop:anyOrder}
For every translatable semigroup $T$ and every integer $h \ge 2$, there exists a basis of $T$ of order $h$.
\end{proposition}
When $T$ is a group, this was proved by Lambert, Plagne and the third author \cite[Theorem 1]{llp}. Our proof makes use of their result.
\begin{proof}
We may assume  that $T$ has a neutral element $0$.
Indeed, supposing that $T$ does not have a neutral element, there exists $x\in T\setminus (-T)$; then $A$ is a basis of 
$T\cup\{0\}$ if and only if $A+x\subset T$ is a basis of $T$, and $\ord^*_T(A+x) = \ord^*_{T\cup\{0\}} (A)$.
We shall construct an infinite sequence $\Lambda=(\Lambda_i)_{i\ge 0}$ of subsets of $T$ such that $\{0\}\subsetneq \Lambda_i$ for every $i\ge 0$ and for any $x\in T$, there exists a unique sequence $\lambda(x)=(\lambda_i(x))_{i\ge 0}$ of finite support such that $x=\sum_{i=0}^{\infty}\lambda_i(x)$ where $\lambda_i(x) \in \Lambda_i$. (The support $\supp(s)$ of a sequence $s=(s_i)_{i\ge 0}\in T^\N$ is the set $\{j\in \N: s_j\neq 0\}$.) As shown in \cite[Proposition 1]{llp} (the arguments there do not use the group structure, only the semigroup structure), such a sequence $\Lambda$ gives rise to a basis of order exactly $h$.

Either $T$ is a group, in which case we can use \cite[Theorem 1]{llp}; or there is a finite subset $\{0\}\subset R\subset T$ and $x\in T$ such that any $t\in T$ may be uniquely written as $t=r+kx$ for some $(r,k)\in R\times \N$.
Let $n=\sum_{i=0}^{\infty}a_i(n)2^i$ be the unique binary decomposition of any integer $n\in\N$, where $a_i(n) \in \{0,1\}$; then we set $\Lambda_i=\{0,2^{i-1}x\}$ for any $i\ge 1$, and $\Lambda_0=R$ if $R\neq\{0\}$, and $\Lambda_i=\{0,2^{i}x\}$ for any $i\ge 0$ otherwise.
The sequence $\Lambda$ has then the desired property.
\end{proof}

\section*{Acknowledgements}
The first author is grateful to the Max-Planck Institute of Mathematics in Bonn, where he put the finishing touch to this work, for its hospitality and financial support. Part of this research was carried out at the Institut de Math\'ematiques de Jussieu - Paris Rive Gauche, and the first and third authors would like to thank the Institut and the team Combinatoire et Optimisation for their hospitality.

\section*{Funding}
The first author was supported by the LABEX MILYON (ANR-10-LABX-0070) of Universit\'e de Lyon, within the program "Investissements d'Avenir" (ANR-11-IDEX-0007) operated by the French National Research Agency (ANR). The third author was supported by National Science Foundation Grant DMS-1702296 and a summer research grant of the College of Liberal Arts of the University of Mississippi.


\begin{thebibliography}{10}

\bibitem{BH}
P.-Y. Bienvenu and F. Hennecart, \emph{Kneser's Theorem in 
$\sigma$-finite abelian groups}, Canad. Math. Bull. {\bf 65} (2022), no. 4, 936--942.

\bibitem{cp}
J. Cassaigne and A. Plagne, \emph{Grekos' $S$ function has a linear growth}, Proc. Amer. Math. Soc. {\bf 132} (2004), 2833--2840.

\bibitem{cassels}
J. W. S. Cassels,
\textit{\"Uber Basen der nat\"urlichen Zahlenreihe},
Abh. Math. Sem. Univ. Hamburg 21 (1957), 247--257.

\bibitem{landmark}
D. Choimet and H. Queff\'elec, \textbf{Twelve Landmarks of Twentieth-Century Analysis}, Cambridge University Press, 2015.

\bibitem{chou}
C. Chou, \emph{The exact cardinality of the set of invariant means on a group}, Proc. Amer. Math. Soc. {\bf 55} (1976), 103--106.

\bibitem{df}
B. Deschamps and B. Farhi,
\emph{Essentialit\'e dans les bases additives}, J. Number Theory {\bf 123} (2007), no. 1, 170--192.

\bibitem{dg}
B. Deschamps and G. Grekos, \emph{Estimation du nombre d'exceptions \`a ce qu'un ensemble de base priv\'e d'un point reste un ensemble de base}, 
J. Reine Angew. Math. {\bf 539} (2001), 45--53.

\bibitem{eg2}
P. Erd\H{o}s and R. L. Graham, \textbf{Old and new problems and results in combinatorial number theory}, Monogr. Enseign. Math. {\bf 28}, 1980.

\bibitem{eg}
P. Erd\H{o}s and R. L. Graham, \emph{On bases with an exact order}, Acta Arith. {\bf 37} (1980), 201--207.

\bibitem{grekos0} 
G. Grekos, \emph{Quelques aspects de la th\'eorie additive des nombres}, doctoral thesis, Universit\'e de Bordeaux I, France (1982).

\bibitem{grekos2}
G. Grekos, \emph{Sur l'ordre d'une base additive}, S\'eminaire de Th\'eorie des Nombres de Bordeaux (Talence, 1987--1988), Exp. No. 31, 13 pp.

\bibitem{grekos}
G. Grekos, \emph{Extremal problems about asymptotic bases: a survey}, 
Integers {\bf 7} (2007), $\#$A16.

\bibitem{Griesmer} J. Griesmer, \emph{Small-sum pairs for upper Banach density in countable abelian groups}, Adv. Math. {\bf 246} (2013), 220--264.

\bibitem{groupprops} Groupprops, \emph{Subgroup lattice and quotient lattice of finite abelian group are isomorphic}, \burl{https://groupprops.subwiki.org/wiki/Subgroup\_lattice\_and\_quotient\_lattice\_of\_finite\_abelian\_group\_are\_isomorphic}

\bibitem{hr} 
Y. O. Hamidoune and {\O}. J. R{\o}dseth, \emph{On bases for $\sigma$-finite groups}, Math. Scand. {\bf 78} (1996), no. 2, 246--254.

\bibitem{hegarty1}
P. Hegarty, \emph{Essentialities in additive bases}, Proc. Amer. Math. Soc. {\bf 137} (2009), no. 5, 1657--1661. 

\bibitem{hegarty2}
P. Hegarty, \emph{The postage stamp problem and essential subsets in integer bases}, in: David and Gregory Chudnovsky (eds.), \textbf{Additive Number Theory : Festschrift in Honor of the Sixtieth Birthday of Melvyn B. Nathanson}, pp. 153--170, Springer-Verlag, New York 2010. 

\bibitem{jia}
X. Jia, \emph{On the exact order of asymptotic bases and bases for finite cyclic groups}, 
in: David and Gregory Chudnovsky (eds.), \textbf{Additive Number Theory : Festschrift in Honor of the Sixtieth Birthday of Melvyn B. Nathanson}, pp. 179--193, Springer-Verlag, New York, 2010. 

\bibitem{llp}
V. Lambert, T. H. L\^e and A. Plagne, \emph{Additive bases in groups}, Israel J. Math. {\bf 217} (2017), no. 1, 383--411.

\bibitem{nn}
J. C. M. Nash and M. B. Nathanson, 
\emph{Cofinite subsets of asymptotic bases for the positive integers},
J. Number Theory {\bf 20} (1985), no. 3, 363--372. 

\bibitem{nathanson2}
M. B. Nathanson, \emph{The exact order of subsets of additive bases}, \textbf{Number theory (New York,
1982)}, volume 1052 of Lecture Notes in Math., pages 273--277. Springer, Berlin, 1984.

\bibitem{p2}
A. Plagne, \emph{A propos de la fonction $X$ d'Erd\H{o}s et Graham},  Ann. Inst. Fourier (Grenoble) {\bf 54} (2004), no. 6, 1717--1767.

\bibitem{p}
A. Plagne, \emph{Problemas combinatorios sobre bases aditivas},  Gac. R. Soc. Mat. Esp. {\bf 9} (2006), 191--201.

\bibitem{p3}
A. Plagne, \emph{Sur le nombre d'\'el\'ements exceptionnels d'une base additive}, J. Reine Angew. Math. {\bf 616} (2008), 47--65.

\bibitem{tv} 
T. Tao and V. H. Vu,
\emph{Additive Combinatorics}, 
Cambridge Studies in Advanced Mathematics {\bf 105} (2006), Cambridge Press University.

\end{thebibliography}
\end{document}